\documentclass[DIV=12]{scrartcl}
\usepackage[english]{babel}
\usepackage{amsmath,amsthm,algorithmic}
\usepackage{txfonts,mathtools}
\usepackage{subcaption, hyperref}
\usepackage{tikz}
\usepackage{hhline}
\usetikzlibrary{calc}
\hypersetup{ colorlinks=false, pdfborder={0 0 0} }
\renewcommand\arraystretch{1.5}

\theoremstyle{definition}
\newtheorem{alg}{Algorithm}[section]
\newtheorem{df}[alg]{Definition}
\newtheorem{ex}[alg]{Example}
\theoremstyle{remark}
\newtheorem*{rem}{Remark}
\theoremstyle{plain}
\newtheorem{lma}[alg]{Lemma}
\newtheorem{thm}[alg]{Theorem}
\newtheorem{prp}[alg]{Proposition}
\newtheorem{crl}[alg]{Corollary}

\DeclareMathOperator \supp {supp}
\DeclareMathOperator \xx {\pmb{\mathrm x}}
\DeclareMathOperator \yy {\pmb{\mathrm y}}
\DeclareMathOperator \zz {\pmb{\mathrm z}}
\DeclareMathOperator \XX {\pmb{\mathrm X}}
\DeclareMathOperator \YY {\pmb{\mathrm Y}}
\DeclareMathOperator \ZZ {\pmb{\mathrm Z}}
\DeclareMathOperator*\tcup {{\textstyle\bigcup}}

\DeclareMathOperator \Dist {Dist}
\DeclareMathOperator \subdiv {subdiv}
\DeclareMathOperator \child {child}
\DeclareMathOperator \midp {mid}

\newcommand \mathbox [2][2em]{\makebox[#1]{$\displaystyle #2$}}
\newcommand \Stackrel [3][1.4ex]{\mathrel{\mathbox[#1]{\stackrel{#2}{#3}}}}
\newcommand \mbigcup [2][2em]{\mathbox[#1]{\bigcup_{#2}}}
\newcommand \sei \coloneqq
\newcommand \N {\mathcal N}
\newcommand \Nx {\N_\mathrm x}
\newcommand \Ny {\N_\mathrm y}
\newcommand \Nz {\N_\mathrm z}
\renewcommand \S {\mathcal S}
\newcommand \Sx {\S_\mathrm x}
\newcommand \Sy {\S_\mathrm y}
\newcommand \Sz {\S_\mathrm z}
\newcommand \R {\mathcal R}
\newcommand \Rx {\R_\mathrm x}
\newcommand \Ry {\R_\mathrm y}
\newcommand \Rz {\R_\mathrm z}
\newcommand \AR {\mathcal{AR}}

\newcommand \barOmega {\overline{\rule{0pt}{8pt}\smash\Omega}}
\newcommand \raiseqed {\\[-3.25em]}
\newcommand \patch [3][\mathbf p,m]{#2^{#1}(#3)}
\newcommand \clos [1][\mathbf p,m]{\operatorname{clos}^{#1}}
\newcommand \D [1][\mathbf p,m]{\operatorname{\mathbf D}^{#1}}
\newcommand \U [1][\mathbf p,m]{U^{#1}}
\newcommand \refine [1][\mathbf p,m]{\operatorname{ref}^{#1}}
\newcommand \M {\mathcal M}
\newcommand \KM {\tilde K}
\newcommand \G {\mathcal G}
\newcommand \Guni [1]{\mathcal G_{u\mid #1}}
\newcommand \A[1][\mathbf p,m] {\mathbb A^{#1}}
\newcommand \AS[1][\mathbf p] {\mathbb{A\negmedspace S}^{#1}}
\newcommand \DC[1][\mathbf p] {\mathbb{D\negmedspace C}^{#1}}
\newcommand \mn {_{\mathrm m}}
\newcommand \mx {_{\mathrm M}}
\newcommand \Mtilde {\overset{\mspace{9mu}\textstyle\sim}{\smash{\M}\rule{0pt}{.8ex}}}
\newcommand \tXi {\overset{\textstyle\sim}{\smash{\Xi}\rule{0pt}{.8ex}}}
\newcommand \ceilfrac [2] {\mathchoice{\bigl\lceil\tfrac{#1}{#2}\bigr\rceil}{\lceil\frac{#1}{#2}\rceil}{\lceil\frac{#1}{#2}\rceil}{\lceil\frac{#1}{#2}\rceil}}

\newcommand \de {\,\mathrm d}
\newcommand \xix {\Xi_\mathrm x}
\newcommand \xiy {\Xi_\mathrm y}
\newcommand \xiz {\Xi_\mathrm z}
\newcommand \hatxix {\hat{\rule{0pt}{8pt}\Xi}_\mathrm x}
\newcommand \hatxiy {\hat{\rule{0pt}{8pt}\Xi}_\mathrm y}
\newcommand \hatxiz {\hat{\rule{0pt}{8pt}\Xi}_\mathrm z}
\newcommand \reell {\mathbb R}
\newcommand \true {\ensuremath{\mathsf{true}}}
\newcommand \false {\ensuremath{\mathsf{false}}}
\newlength \eqheight
\newcommand \EQ {\mathrel{\rule{0pt}{\eqheight}=}}

\newcounter{num}
\numberwithin{num}{alg}
\renewcommand{\thenum}{\arabic{num}}
\newcommand{\pnumpx}[1][]{\refstepcounter{num}\noindent\textbf{(\thenum)}\enspace\emph{#1}}

\newcommand{\numref}[1]{\textbf{(\ref{#1})}}

\newcommand{\email}[1]{\href{mailto:#1}{#1}}

\newenvironment{silentproof}{\pushQED{\qed}}{\popQED}

\title{Globally structured 3D analysis-suitable T-splines:  definition, linear independence and  $m$-graded local refinement}
\author{Philipp Morgenstern\footnote{Rheinische Friedrich-Wilhelms-Universit\"at Bonn \newline Wegelerstr.\ 6, 53115 Bonn, Germany / +49\,228\,73-60153  / \email{morgenstern@ins.uni.bonn.de} \vspace{1ex} \newline 
The author gratefully acknowledges support by the Deutsche Forschungsgemeinschaft in the Priority Program 1748 
``Reliable simulation techniques in solid mechanics: Development of non-standard discretization methods, mechanical and mathematical analysis'' under the project ``Adaptive isogeometric modeling of propagating strong discontinuities in heterogeneous materials''. }}

\begin{document}
\maketitle
\begin{abstract}
This paper addresses the linear independence of T-splines that correspond to refinements of three-dimensional tensor-product meshes. We give an abstract definition of analysis-suitability, and prove that it is equivalent to dual-compatibility, wich guarantees linear independence of the T-spline blending functions. In addition, we present a local refinement algorithm that generates analysis-suitable meshes and has linear computational complexity in terms of the number of marked and generated mesh elements.
\end{abstract}

\section{Introduction}
T-splines \cite{SZBN:2003} have been introduced as a free-form geometric technology and were the first tool of interest in Adaptive Isogeometric Analysis (IGA). Although they are still among the most common techniques in Computer Aided Design, T-splines provide algorithmic difficulties that have motivated a wide range of alternative approaches to mesh-adaptive splines, such as hierarchical B-splines \cite{Forsey:Bartels:1988,KVZB:2014}, THB-splines \cite{GJS:2012}, LR splines \cite{DLP:2013}, hierarchical T-splines \cite{ESLT:2015}, amongst many others. 

One major difficulty using T-splines for analysis has been pointed out by Buffa, Cho and Sangalli \cite{BCS:2010},  who showed that general T-spline meshes can induce linear dependent T-spline blending functions. This prohibits the use of T-splines as a basis for analytical purposes such as solving a discretized partial differential equation. 
This insight motivated the research on T-meshes that guarantee the linear independence of the corresponding T-spline blending functions, referred to as \emph{analysis-suitable T-meshes}.
Analysis-suitability has been characterized in terms of topological mesh properties \cite{ZSHS:2012} and, in an alternative approach, through the equivalent concept of Dual-Compatibility \cite{BBCS:2012}.
While Dual-Compatibility has been characterized in arbitrary dimensions \cite{BBSV:2014}, Analysis-Suitability as a topological criterion for linear independence of the T-spline functions is only available in the two-dimensional setting.

In this paper, we introduce analysis-suitable T-splines for those 3D meshes that are refinements of tensor-product meshes, and propose an algorithm for their local refinement, based on our preliminary work in \cite{Morgenstern:Peterseim:2015}. In addition, we generalize the algorithm from \cite{Morgenstern:Peterseim:2015} by introducing a grading parameter $m$ that represents the number of children in a single elements' refinement. This allows the user to fully control how local the refinement shall be. Choosing $m$ large yields meshes with very local refinement, while a small $m$ will cause more wide-spreaded refinement. The former yields a smaller number of degrees of freedom, while the latter reduces the overlap of the basis functions and hence provides sparser Galerkin and collocation matrices.

This paper is organized as follows. Section~\ref{sec: refinement} defines the initial mesh and basic refinement steps and introduces our new refinement algorithm. Section~\ref{sec: adm meshes} then characterizes the class of `admissible meshes' generated by this algorithm.
In Section~\ref{sec: Tspline def} we give a brief definition of trivariate odd-degree T-splines.
In Section~\ref{sec: AS} we give an abstract definition of Analysis-Suitability in the 3D setting and prove that all admissible meshes are analysis-suitable. In Section~\ref{sec: DC} we define dual-compatible meshes, and prove that  analysis-suitability and dual-compatibility are equivalent, and that all dual-compatible meshes provide linear independent T-spline functions. (Figure~\ref{fig: overview} illustrates this ``long way'' to linear independence.) Section~\ref{sec: complexity} proves linear complexity of the refinement procedure, and conclusions and an outlook to future work are finally given in Section~\ref{sec: conclusions}.
\begin{figure}[ht]
\centering
\begin{tabular}{l@{\quad}c@{\quad}c}
& \small Symbol & \small Section \\\hline
refinement algorithm & $\refine$ & \ref{sec: refinement} \\
admissible meshes & $\A$ & \ref{sec: adm meshes} \\
analysis-suitable meshes & $\AS$ & \ref{sec: AS} \\
dual-compatible meshes & $\DC$ & \ref{sec: DC}\\\hline
\end{tabular}\bigskip
\[
\refine(\A)\stackrel{\text{Theorem~\ref{thm: ref works}}}\subseteq
\A\stackrel{\text{Theorem~\ref{thm: A in AS}}}\subseteq
\AS\stackrel{\text{Theorem~\ref{thm: AS in DC}}}\EQ
\DC\stackrel{\text{Theorem~\ref{thm: DC has dual basis}}}\subseteq
\left[\parbox{3.25cm}{\centering \small meshes with linearly independent T-splines}\right]
\]
\caption{How we prove linear independence of the T-splines induced by the generated meshes.}
\label{fig: overview}
\end{figure}

\section{Adaptive mesh refinement}\label{sec: refinement}
This section defines the new refinement algorithm and characterizes the class of meshes which are generated by this algorithm.
The algorithm is essentially a 3D version of the one introduced in \cite{Morgenstern:Peterseim:2015}, with the additional feature of variable grading.
The initial mesh is assumed to have a very simple structure. In the context of IGA, the partitioned rectangular domain is referred to as \emph{index domain}. This is, we assume that the \emph{physical domain} (on which, e.g., a PDE is to be solved) is obtained by a continuous map from the active region (cf.\ Section~\ref{sec: DC}), which is a subset of the index domain. Throughout this paper, we focus on the mesh refinement only, and therefore we will only consider the index domain. For the parametrization and refinement of the T-spline blending functions, we refer to \cite{SLSH:2012}.

\begin{df}[Initial mesh, element]
Given $\tilde X,\tilde Y,\tilde Z\in\mathbb N$, the initial mesh $\G_0$ is a tensor product mesh consisting of closed cubes (also denoted \emph{elements}) with side length 1, i.e.,
\[\G_0\sei\Bigl\{[x-1,x]\times[y-1,y]\times[z-1,z]\mid x\in\{1,\dots,\tilde X\},y\in\{1,\dots,\tilde Y\},z\in\{1,\dots,\tilde Z\}\Bigr\}.\]
The domain partitioned by $\G_0$ is denoted by $\Omega\sei (0,\tilde X)\times (0,\tilde Y)\times(0,\tilde Z)$.
\end{df}

The key property of the refinement algorithm will be that refinement of an element $K$ is allowed only if elements in a certain neighbourhood are sufficiently fine. The size of this neighbourhood, which is denoted $(\mathbf p,m)$-patch and defined through the definitions below, depends on the size of $K$, the polynomial degree $\mathbf p=(p_1,p_2,p_3)$ of the T-spline functions, and the grading parameter $m$. For the sake of legibility, we assume that $p_1,p_2,p_3$ are odd and greater or equal to 3. (For comments on even polynomial degrees, see Section~\ref{sec: conclusions}.)

\begin{df}[Level]
The \emph{level} of an element $K$ is defined by \[\ell(K)\sei-\log_m|K|,\] where $m$ is the manually chosen grading parameter, i.e., the number of children in a single elements' refinement, and $|K|$ denotes the volume of $K$.
This implies that all elements of the initial mesh have level zero and that the refinement of an element $K$ yields $m$ elements of level $\ell(K)+1$.
\end{df}

\begin{df}[Vector-valued distance]\label{df: distance}
Given $x\in\barOmega$ and an element $K$, we define their distance 
as the componentwise absolute value of the difference between $x$ and the midpoint of $K$,
\begin{align*}
\Dist(K,x)&\sei\operatorname{abs}\bigl(\midp(K)-x\bigr)\ \in\reell^3,\\
\text{with}\quad\operatorname{abs}(y) &\sei \bigl(\lvert y_1\rvert, \lvert y_2\rvert, \lvert y_3\rvert\bigr).
\end{align*}
For two elements $K_1,K_2$, we define the shorthand notation 
\[\Dist(K_1,K_2)\sei\operatorname{abs}\bigl(\midp(K_1)-\midp(K_2)\bigr).\]
\end{df}

\begin{df}\label{df: magic patch}
Given an element $K$, a grading parameter $m\ge2$ and the polynomial degree $\mathbf p=(p_1,p_2,p_3)$, 
we define the open environment
\begin{align*}
\U(K)&\sei\{x\in\Omega\mid\Dist(K,x)<\D(\ell(K))\},
\shortintertext{where}
\D(k)&\sei\begin{cases}m^{-k/3}\,\bigl(p_1+\tfrac32,p_2+\tfrac32,p_3+\tfrac32\bigr)&\text{if }k=0\bmod3,
\\[.7ex]
m^{-(k-1)/3}\,\bigl(\tfrac{p_1+3/2}m,p_2+\tfrac32,p_3+\tfrac32\bigr)&\text{if }k=1\bmod3,
\\[.7ex]
m^{-(k-2)/3}\,\bigl(\tfrac{p_1+3/2}m,\tfrac{p_2+3/2}m,p_3+\tfrac32\bigr)&\text{if }k=2\bmod3.\end{cases}
\intertext{The $(\mathbf p,m)$-patch of $K$ is defined as the set of all elements that intersect with environment of $K$,}
\patch\G K &\sei \{K'\in\G\mid  K'\cap\U(K)\neq\emptyset\}.
\end{align*}
Note as a technical detail that this definition does \emph{not} require that $K\in\G$. See also Figure~\ref{fig: magic patch examples} for examples.
\end{df}

\begin{rem}
By definition, the size of the $(\mathbf p,m)$-patch of an element $K$ scales linearly with the size of $K$ and with the polynomial degree $\mathbf p$. Since $\D(k)$ is decreasing in $m$, choosing $m$ large will cause small $(\mathbf p,m)$-patches and hence more localized refinement.
\end{rem}

\begin{figure}[ht]
\centering
\includegraphics[width=.3\textwidth]{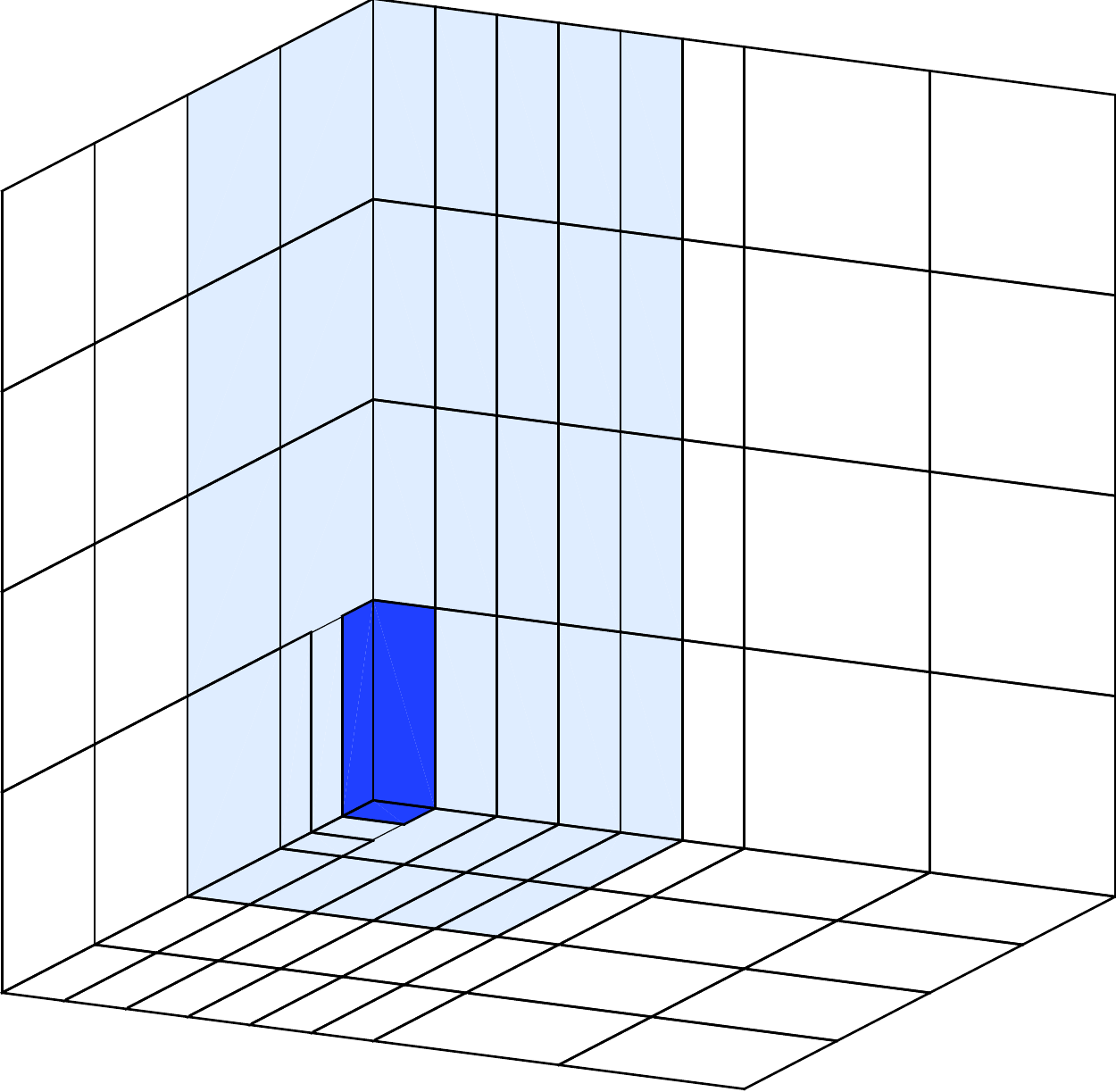}%
\hspace{.049\textwidth}%
\includegraphics[width=.3\textwidth]{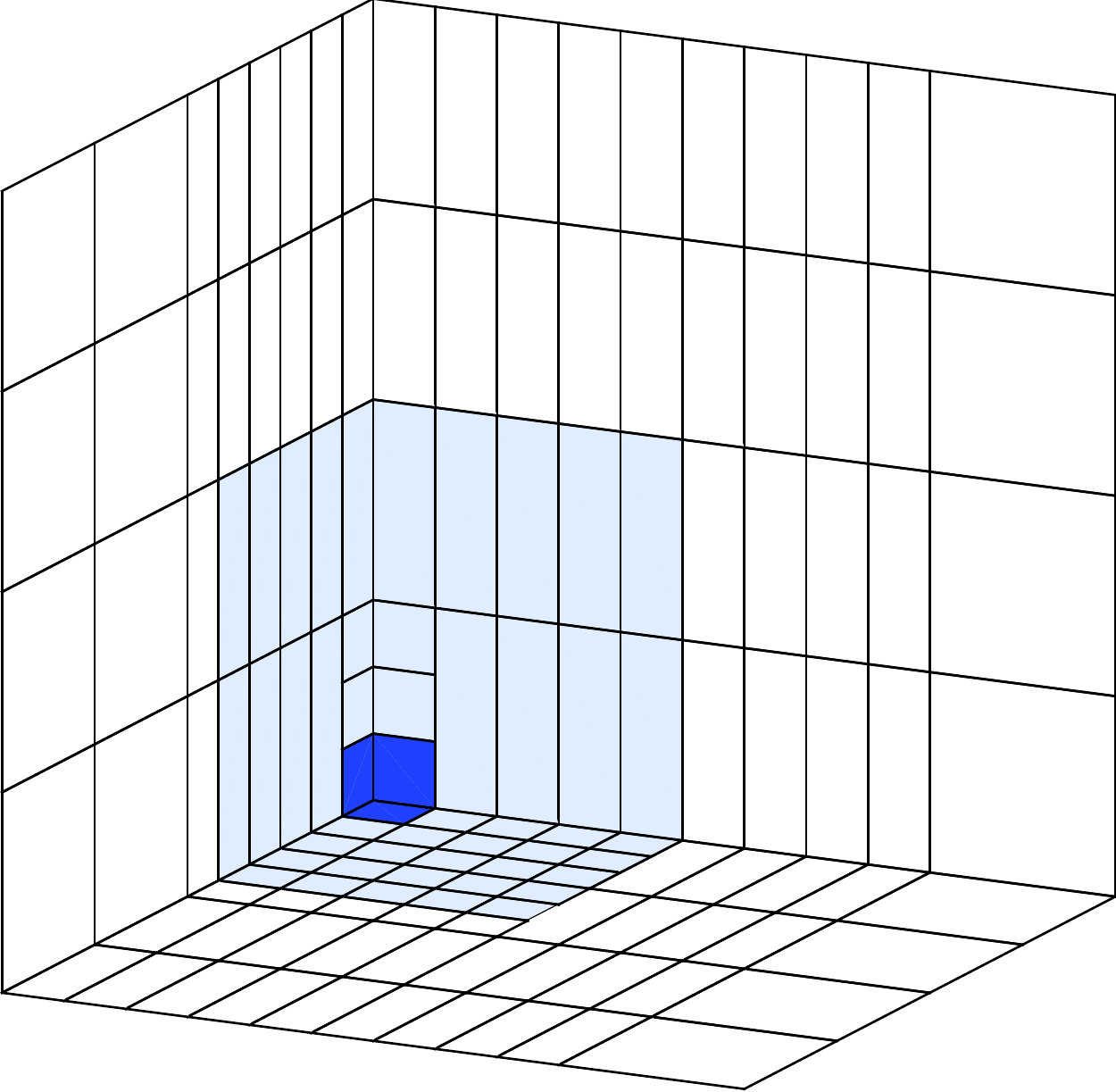}%
\hspace{.049\textwidth}%
\includegraphics[width=.3\textwidth]{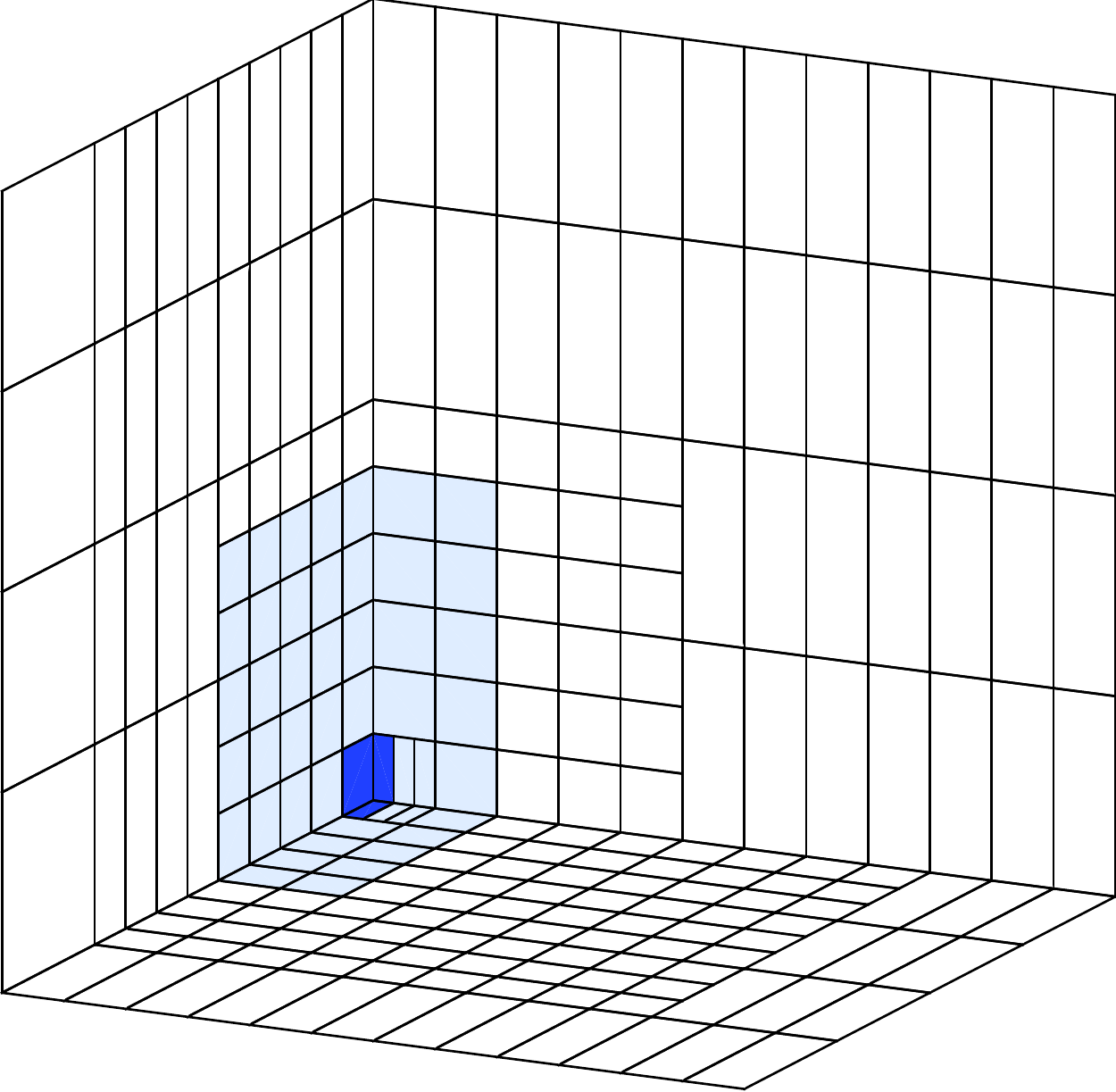}%
\caption{Examples for the $(\mathbf p,m)$-patch of an element $K$, for $\mathbf p=(3,3,3)$, $m=3$ and $\ell(K)=2,3,4$.}
\label{fig: magic patch examples}
\end{figure}

In the subsequent definitions, we will give a detailed description of the elementary subdivision steps and then present the new refinement algorithm. 

\begin{df}[Subdivision of an element]
Given an arbitrary element $K=[x,x+\tilde x]\times[y,y+\tilde y]\times[z,z+\tilde z]$, where $x, y,z, \tilde x,\tilde y,\tilde z\in\mathbb{R}$ and $\tilde x,\tilde y,\tilde z>0$, we define the operators
\begin{align*}
\subdiv_\mathrm x(K) &\sei \bigl\{\,[x+\tfrac {j-1}m\tilde x,x+\tfrac jm\tilde x]\times[y,y+\tilde y]\times[z,z+\tilde z]\mid j\in\{1,\dots,m\}\bigr\},\\
\enspace\subdiv_\mathrm y(K) &\sei \bigl\{\,[x,x+\tilde x]\times[y+\tfrac {j-1}m\tilde y,y+\tfrac jm\tilde y]\times[z,z+\tilde z]\mid j\in\{1,\dots,m\}\bigr\},
\\\text{and}\enspace
\enspace\subdiv_\mathrm z(K) &\sei \bigl\{\,[x,x+\tilde x]\times[y,y+\tilde y]\times[z+\tfrac {j-1}m\tilde z,z+\tfrac jm\tilde z]\mid j\in\{1,\dots,m\}\bigr\}.
\end{align*}
These operators will be used for $x$-, $y$-, and $z$-orthogonal subdivisions in the refinement procedure. Their output is illustrated in Figure~\ref{fig: elemref}.
\end{df}
\begin{df}[Subdivision]\label{df: subdivision}%
Given a mesh $\G$ and an element $K\in\G$, we denote by $\subdiv(\G,K)$ the mesh that results from a level-dependent subdivision of $K$,
\begin{align*}
\subdiv(\G,K)&\sei\G\setminus\{K\}\cup\child(K),\\
\text{with}\enspace\child(K)&\sei
\begin{cases}\subdiv_\mathrm x(K)&\text{if }\ell(K)=0\bmod3,\\\subdiv_\mathrm y(K)&\text{if }\ell(K)=1\bmod3,\\\subdiv_\mathrm z(K)&\text{if }\ell(K)=2\bmod3.\end{cases}
\end{align*}
\begin{figure}[ht]
\centering
\includegraphics{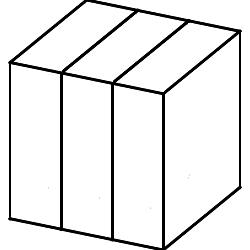}\qquad
\includegraphics{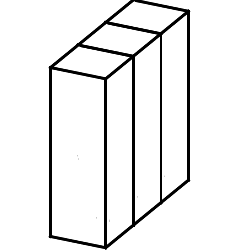}
\includegraphics{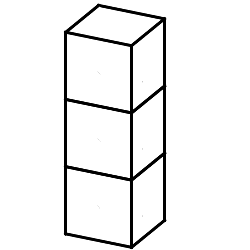}
\caption{Elementary subdivision routines for $m=3$:
$x$-orthogonal subdivision of an element with level 0 (left),
$y$-orthogonal subdivision of an element with level 1 (middle), and
$z$-orthogonal subdivision of an element with level 2 (right).}
\label{fig: elemref}
\end{figure}
\end{df}
\begin{df}[Multiple subdivisions]
We introduce the shorthand notation $\subdiv(\G,\M)$ for the subdivision of several elements $\M=\{K_1,\dots,K_J\}\subseteq\G$, defined by successive subdivisions in an arbitrary order,
\[\subdiv(\G,\M)\sei\subdiv(\subdiv(\dots\subdiv(\G,K_1),\dots),K_J).\]
\end{df}

We will now define the new refinement algorithm through the subdivision of a superset $\clos(\G,\M)$ of the marked elements $\M$. In the remaining part of this section, we characterize the class of meshes generated by this refinement algorithm.

\begin{alg}[Closure]\label{alg: closure}
Given a mesh $\G$ and a set of marked elements $\M\subseteq\G$ to be refined, the \emph{closure} $\clos(\G,\M)$ of $\M$ is computed as follows.
\begin{algorithmic}
\STATE $\Mtilde\sei \M$
\REPEAT
\FORALL{$K\in\Mtilde$}
\STATE $\Mtilde\sei\Mtilde\cup\bigl\{K'\in\patch\G K\mid \ell(K')<\ell(K)\bigr\}$
\ENDFOR
\UNTIL{$\Mtilde$ stops growing}
\RETURN $\clos(\G,\M)=\Mtilde$
\end{algorithmic}
\end{alg}

\begin{alg}[Refinement]\label{alg: refinement}

Given a mesh $\G$ and a set of marked elements $\M\subseteq\G$ to be refined, $\refine(\G,\M)$ is defined by \[\refine(\G,\M)\sei\subdiv(\G,\clos(\G,\M)).\]
An example of this algorithm is given in Figure~\ref{fig: refinement algorithm}.
\end{alg}
\begin{figure}[ht]
\centering\Large
\includegraphics[width=.25\textwidth]{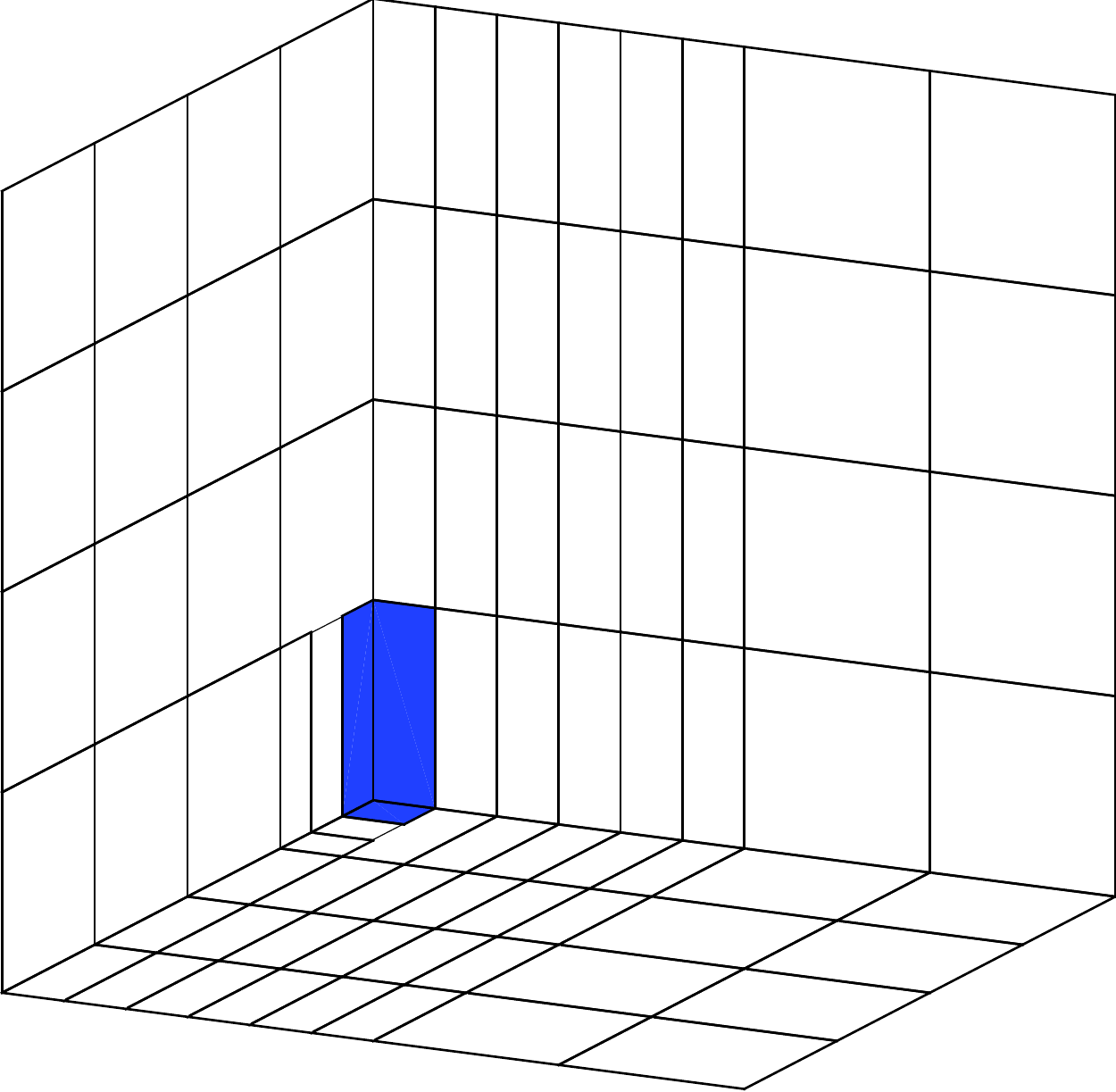}%
\makebox[.125\textwidth]{\raisebox{.125\textwidth}{$\stackrel{1^\text{st}\text{ iter.}}\rightarrow$}}%
\includegraphics[width=.25\textwidth]{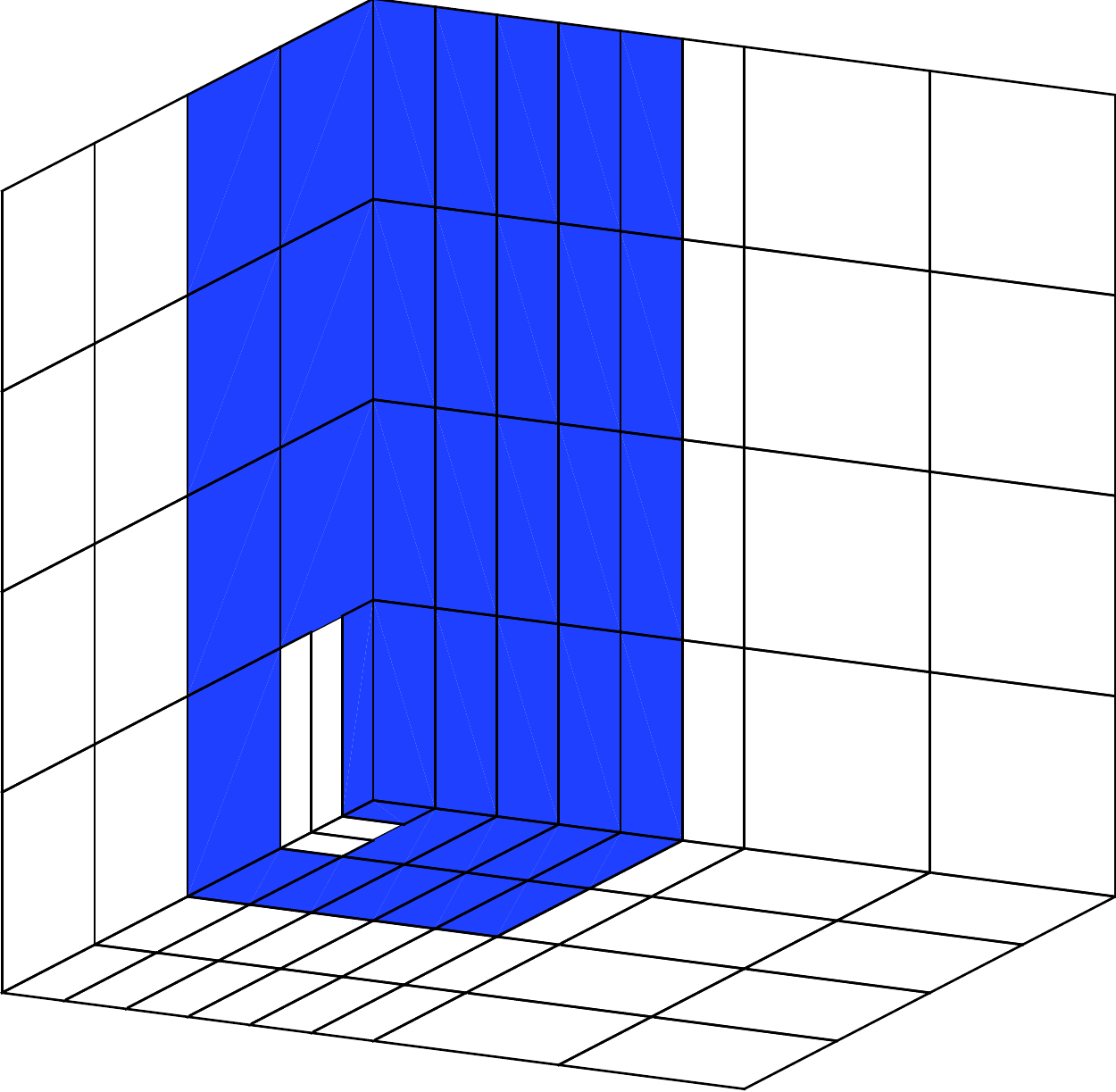}%
\makebox[.125\textwidth]{\raisebox{.125\textwidth}{$\stackrel{2^\text{nd}\text{ iter.}}\rightarrow$}}%
\includegraphics[width=.25\textwidth]{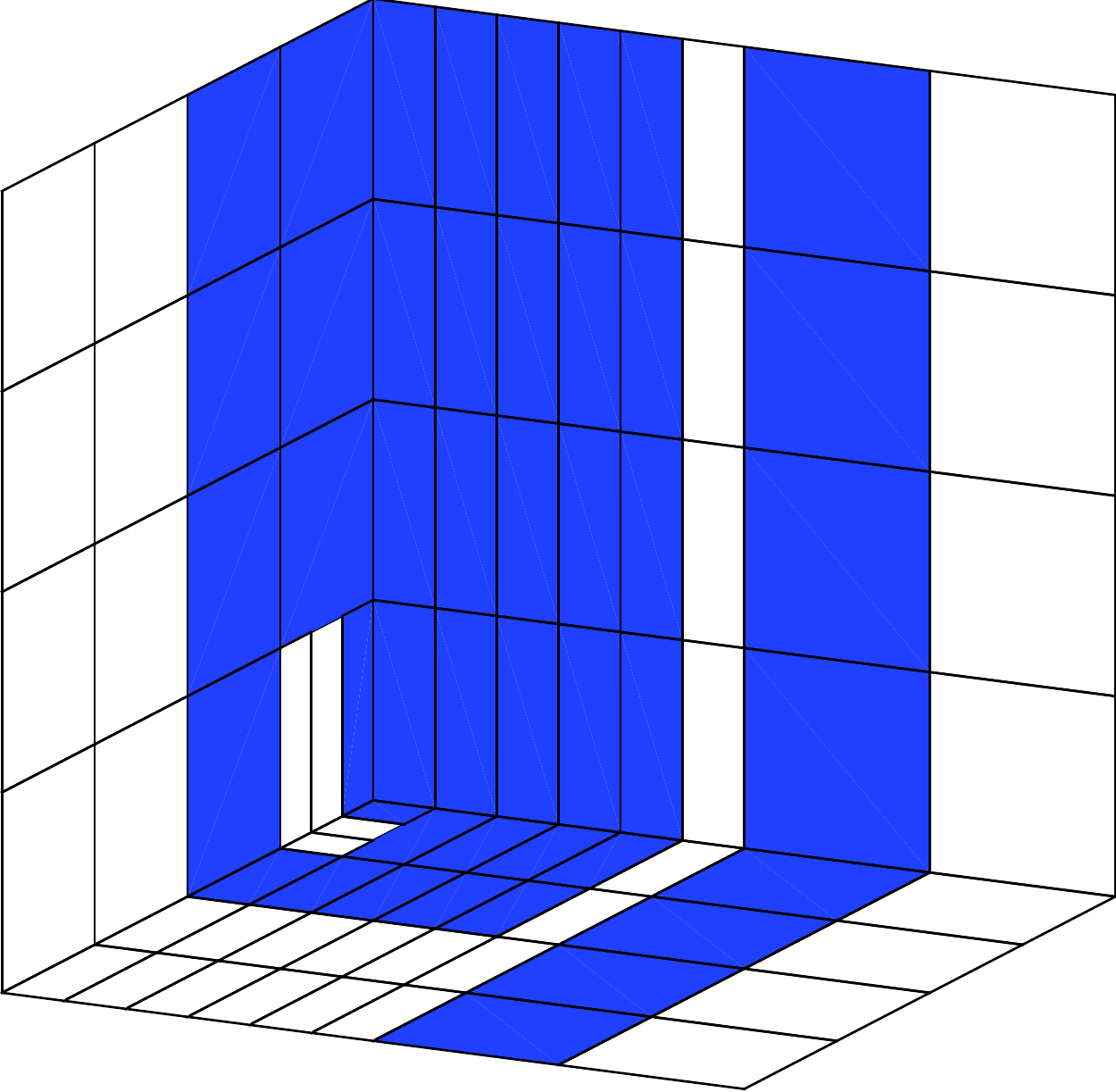}\\
\makebox[.125\textwidth]{\raisebox{.125\textwidth}{$\stackrel{3^\text{rd}\text{ iter.}}\rightarrow$}}%
\includegraphics[width=.25\textwidth]{refex_15}%
\makebox[.125\textwidth]{\raisebox{.125\textwidth}{$\stackrel{\text{subdiv.}}\rightarrow$}}%
\includegraphics[width=.25\textwidth]{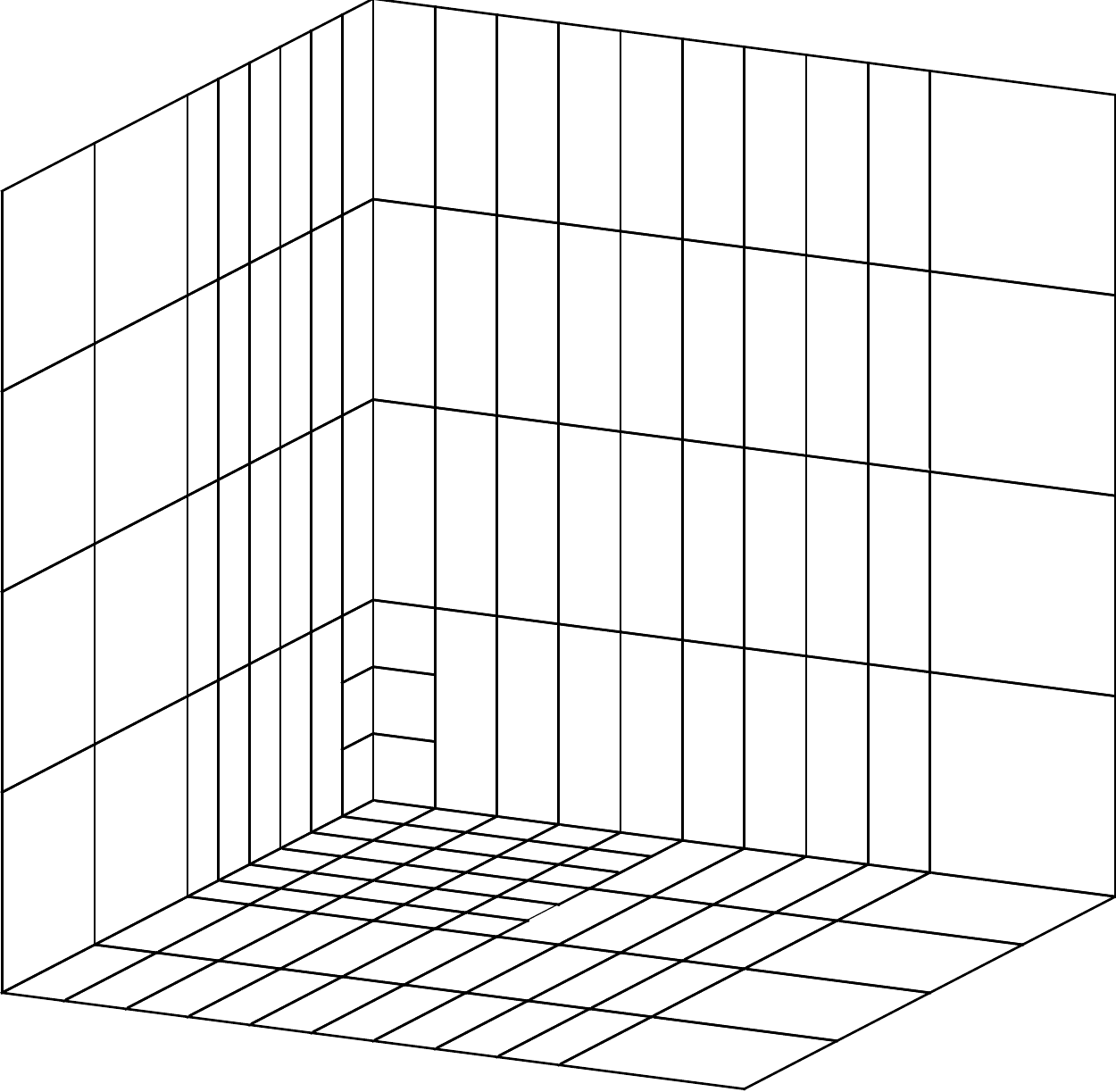}
\caption{Example for Algorithm~\ref{alg: refinement}, with $\mathbf p=(3,3,3)$, $m=3$ and $\M=\{K\}$ with $\ell(K)=2$. In the first iteration of the \textbf{for}-loop, all coarser (level 1) elements in the $(\mathbf p,m)$-patch of $K$ are marked as well. in the second iteration, all coarser (level 0) ``neighbours'' of those elements are also marked. Since there are no elements that are coarser than level 0, the third iteration does not change anything. Hence the \textbf{for}-loop ends, and all marked elements are subdivided in the directions that correspond to their levels. }
\label{fig: refinement algorithm}
\end{figure}

\begin{figure}
\centering
\begin{subfigure}[b]{.3\textwidth}
\includegraphics[width=\textwidth]{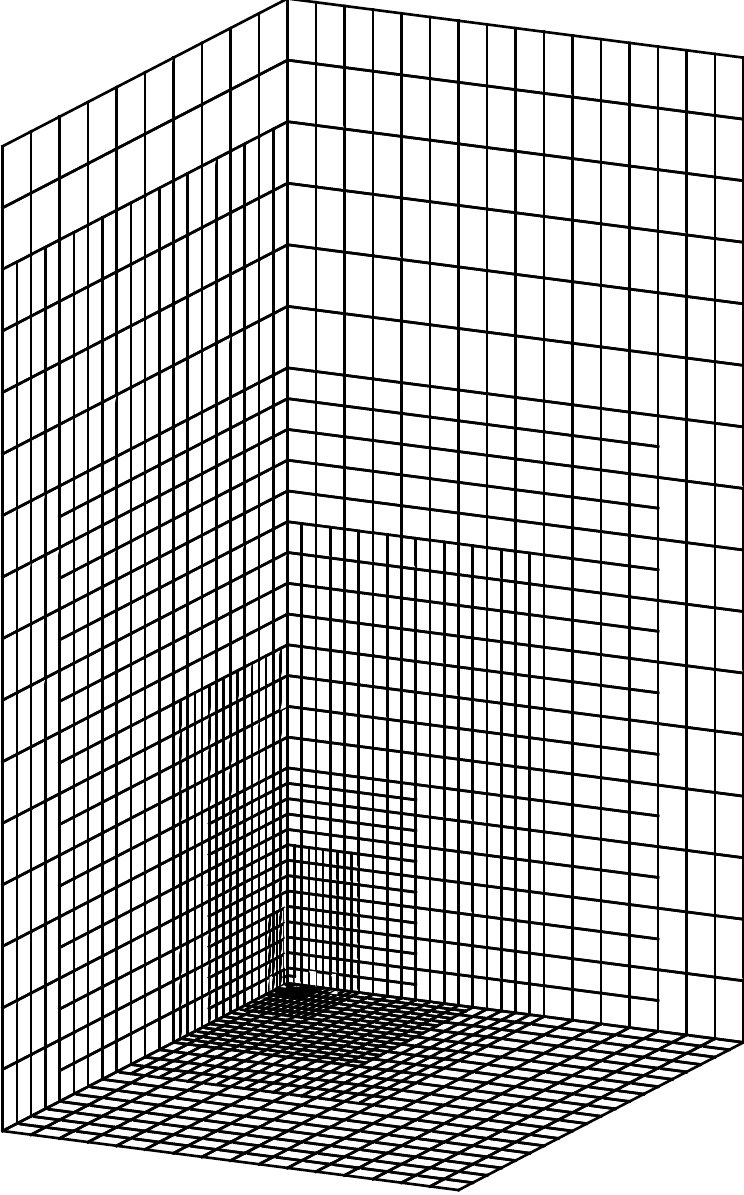}
\caption{$m=2$}\label{sfig: refex 2}
\end{subfigure}\quad
\begin{subfigure}[b]{.3\textwidth}
\includegraphics[width=\textwidth]{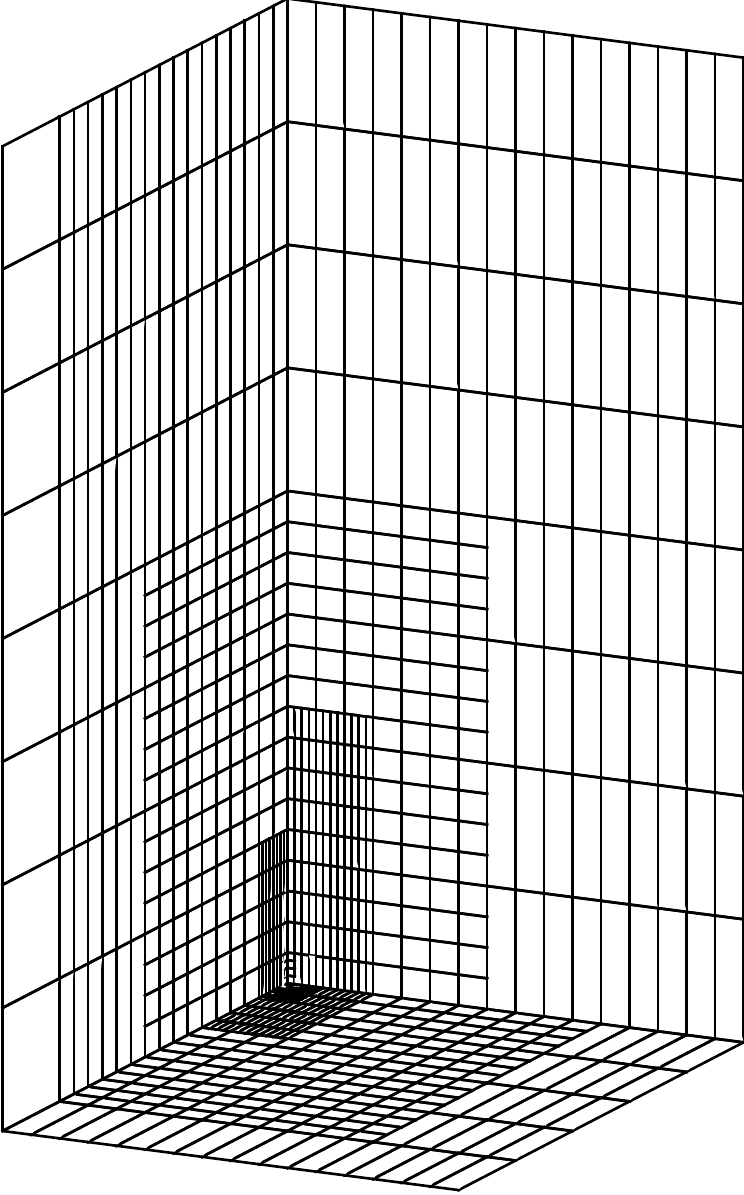}
\caption{$m=4$}\label{sfig: refex 4}
\end{subfigure}\quad
\begin{subfigure}[b]{.3\textwidth}
\includegraphics[width=\textwidth]{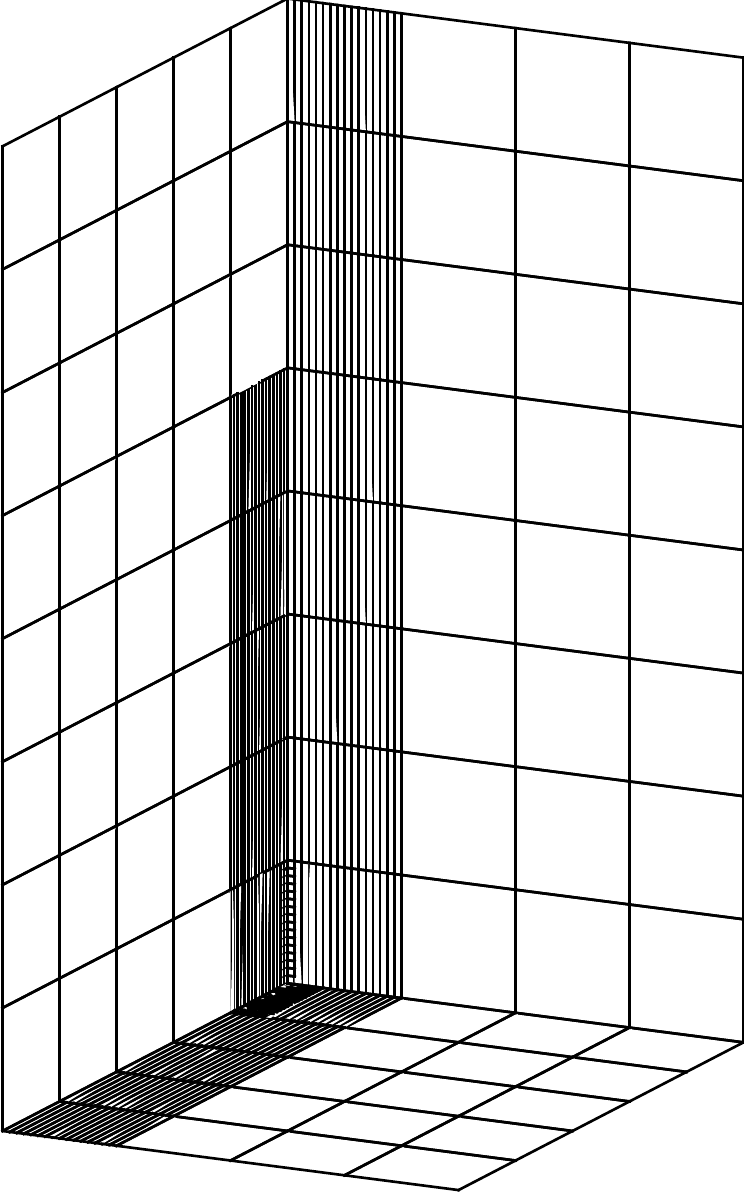}
\caption{$m=16$}\label{sfig: refex 16}
\end{subfigure}
\caption{Refinement examples for $\mathbf p=(3,3,3)$ and different choices of $m$. In all cases, the initial mesh consists of $4\times5\times8$ cubes of size $1\times1\times1$, and is refined by marking the lower left front corner element repeatedly until it is of the size $\tfrac1{16}\times\tfrac1{16}\times\tfrac1{16}$. }
\label{fig: refex}
\end{figure}
\begin{ex}
Consider an initial mesh that consists of $4\times5\times8$ cubes of size $1\times1\times1$.
We refine the mesh by marking the lower left front corner element repeatedly until it is of the size $\tfrac1{16}\times\tfrac1{16}\times\tfrac1{16}$.
The resulting meshes for different choices of $m$ are illustrated in Figure~\ref{fig: refex}, and the results are listed below.
\begin{center}
\begin{tabular}{c|c|c|c}
Figure & $m$ & \parbox{7em}{\centering number of \\ refinement steps} & \parbox{6em}{\centering number of \\ new elements} \\\hline
\ref{sfig: refex 2} & 2 & 12 & 10728 \\
\ref{sfig: refex 4} & 4 & 6 & 3175 \\
\ref{sfig: refex 16} & 16 & 3 & 1030
\end{tabular}
\end{center}
\end{ex}

\section{Admissible meshes}\label{sec: adm meshes}
In the subsequent definitions, we introduce a class of admissible meshes. We will then prove that this class coindices with the meshes generated by Algorithm~\ref{alg: refinement}.

\begin{df}[$(\mathbf p,m)$-admissible subdivisions]\label{df: adm. subdivision}%
Given a mesh $\G$ and an element $K\in\G$, the subdivision of $K$ is called \emph{$(\mathbf p,m)$-admissible} if all $K'\in\patch\G K$ satisfy $\ell(K')\ge\ell(K)$.

In the case of several elements $\M=\{K_1,\dots,K_J\}\subseteq\G$, the subdivision $\subdiv(\G,\M)$ is $(\mathbf p,m)$-admissible if there is an ordering $(\sigma(1),\dots,\sigma(J))$ (this is, if there is a permutation $\sigma$ of $\{1,\dots,J\}$) such that \[\subdiv(\G,\M)=\subdiv(\subdiv(\dots\subdiv(\G,K_{\sigma(1)}),\dots),K_{\sigma(J)})\] is a concatenation of $(\mathbf p,m)$-admissible subdivisions.
\end{df}
\begin{df}[Admissible mesh]\label{df: admissible mesh}
A refinement $\G$ of $\G_0$ is \emph{$(\mathbf p,m)$-admissible} if there is a sequence of meshes $\G_1,\dots,\G_J=\G$ and markings $\M_j\subseteq\G_j$ for $j=0,\dots,J-1$, such that $\G_{j+1}=\subdiv(\G_j,\M_j)$ is an $(\mathbf p,m)$-admissible subdivision for all $j=0,\dots,J-1$. The set of all $(\mathbf p,m)$-admissible meshes, which is the initial mesh and its $(\mathbf p,m)$-admissible refinements, is denoted by $\A$. For the sake of legibility, we write `admissible' instead of `$(\mathbf p,m)$-admissible' throughout the rest of this paper.
\end{df}

\begin{thm}\label{thm: ref works}
Any admissible mesh $\G$ and any set of marked elements $\M\subseteq\G$ satisfy \[\refine(\G,\M)\in\A.\]
\end{thm}
The proof of Theorem~\ref{thm: ref works} given at the end of this section relies on the subsequent results.

\begin{lma}\label{lma: magic patches are nested}
Given an admissible mesh $\G$ and two nested elements $K\subseteq\hat K$ with $K,\hat K\in\tcup\A$, the corresponding $(\mathbf p,m)$-patches are nested in the sense $\patch\G K\subseteq\patch\G{\hat K}$.
\end{lma}
The proof is given in Appendix~\ref{apx: magic patches are nested}.

\begin{lma}[local quasi-uniformity]\label{lma: levels change slowly}
Given  $K\in\G\in\A$, 
any $K'\in\patch\G K$ satisfies $\ell(K')\ge\ell(K)-1$.
\end{lma}
The proof is given in Appendix~\ref{apx: levels change slowly}.

\begin{proof}[Proof of Theorem~\ref{thm: ref works}]
Given the mesh $\G\in\A$ and marked elements $\M\subseteq\G$ to be refined, we have to show that there is a sequence of meshes that are subsequent admissible refinements, with $\G$ being the first and $\refine(\G,\M)$ the last mesh in that sequence.

Set $\Mtilde\sei\clos(\G,\M)$ and
\begin{alignat}{2}
\notag\overline L&\sei\max\ell(\Mtilde),\quad\underline L\sei\min\ell(\Mtilde)&&\\
\notag\M_j&\sei\bigl\{K\in\Mtilde\mid\ell(K)=j\bigr\}&&\text{for}\enspace j=\underline L,\dots,\overline L\\
\label{eq: ref works 1}\G_{\underline L}&\sei\G,\quad \G_{j+1}\sei\subdiv(\G_j,\M_j)&\enspace&\text{for}\enspace j=\underline L,\dots,\overline L.
\end{alignat}
It follows that $\refine(\G,\M)=\G_{\overline L+1}$. We will show by induction over $j$ that all subdivisions in  \eqref{eq: ref works 1} are admissible.

For the first step $j=\underline L$, we know $\{K'\in\Mtilde\mid\ell(K')<\underline L\}=\emptyset$, and by construction of $\Mtilde$ that for each $K\in\Mtilde_{\underline L}$ holds  
$\{K'\in\patch\G K\mid\ell(K')<\ell(K)\}\subseteq\Mtilde$.
Together with $\ell(K)=\underline L$, it follows for any $K\in\Mtilde_{\underline L}$ that  there is no $K'\in\patch\G K$ with $\ell(K')<\ell(K)$. This is, the subdivisions of all $K\in\Mtilde_{\underline L}$ are admissible independently of their order and hence $\subdiv(\G_{\underline L},\Mtilde_{\underline L})$ is admissible.

Consider an arbitrary step $j\in\{\underline L,\dots,\overline L\}$ and assume that $\G_{\underline L},\dots,\G_j$ are admissible meshes. 
Assume for contradiction that there is $K\in\M_j$ of which the subdivision is not admissible, i.e., there exists $K'\in\smash{\patch{\G_j}K}$ with $\ell(K')<\ell(K)$ and consequently $K'\notin\Mtilde$, because $K'$ has not been refined yet. It follows from the closure Algorithm~\ref{alg: closure} that $K'\notin\G$. Hence,  there is $\hat K\in\G$ such that $K'\subset\hat K$. 
We have $\ell(\hat K)<\ell(K')<\ell(K)$, which implies $\ell(\hat K)<\ell(K)-1$. Note that $K\in\G$ because $\M_j\subseteq\Mtilde\subseteq\G$.
From $K'\in\patch{\G_j}K$, it follows by definition that $K'\cap\U(K)\neq\emptyset$, and $K'\subset\hat K$ yields 
$\hat K\cap\U(K)\neq\emptyset$ and hence $\hat K\in\patch\G K$.
Together with $\ell(\hat K)<\ell(K)-1$, Lemma~\ref{lma: levels change slowly} implies that $\G$ is not admissible, which contradicts the assumption.
\end{proof}

\section{T-spline definition}\label{sec: Tspline def}
In this section, we define trivariate T-spline functions corresponding to a given admissible mesh. We roughly follow the definitions from \cite{Morgenstern:Peterseim:2015}.
\begin{df}[Active nodes]
For each element $K=[x,x+\tilde x]\times[y,y+\tilde y]\times[z,z+\tilde z]$, the corresponding set of vertices is denoted by \[\N(K)\sei\{x,x+\tilde x\}\times\{y,y+\tilde y\}\times\{z,z+\tilde z\}.\] We refer to the elements of $\N\sei\bigcup_{K\in\G}\N(K)$ as \emph{nodes}.
We define the \emph{active region} \[\AR\sei\Bigl[\ceilfrac{p_1}2,\tilde X-\ceilfrac{p_1}2\Bigr]\times\Bigl[\ceilfrac{p_2}2,\tilde Y-\ceilfrac{p_2}2\Bigr]\times\Bigl[\ceilfrac{p_3}2,\tilde Z-\ceilfrac{p_3}2\Bigr]\]
and the set of \emph{active nodes} $\N_A\sei\N\cap\AR$. 
\end{df}
\begin{df}[Skeleton]
Given a mesh $\G$, denote the union of all closed $x$-orthogonal element faces by $\xix\sei\tcup_{K\in\G}\xix(K)$, with 
\begin{align*}
\xix(K) &\sei \{x,x+\tilde x\}\times[y,y+\tilde y]\times [z,z+\tilde z]\\
\text{for any } K &= [x,x+\tilde x]\times[y,y+\tilde y]\times [z,z+\tilde z]\in\G.
\end{align*}
We call $\xix$ the \emph{$x$-orthogonal skeleton}. Analogously, we denote the $y$-orthogonal skeleton by $\xiy$, and the $z$-orthogonal skeleton by $\xiz$.
\end{df}

\begin{figure}[ht]
\centering
\includegraphics[width=.275\textwidth]{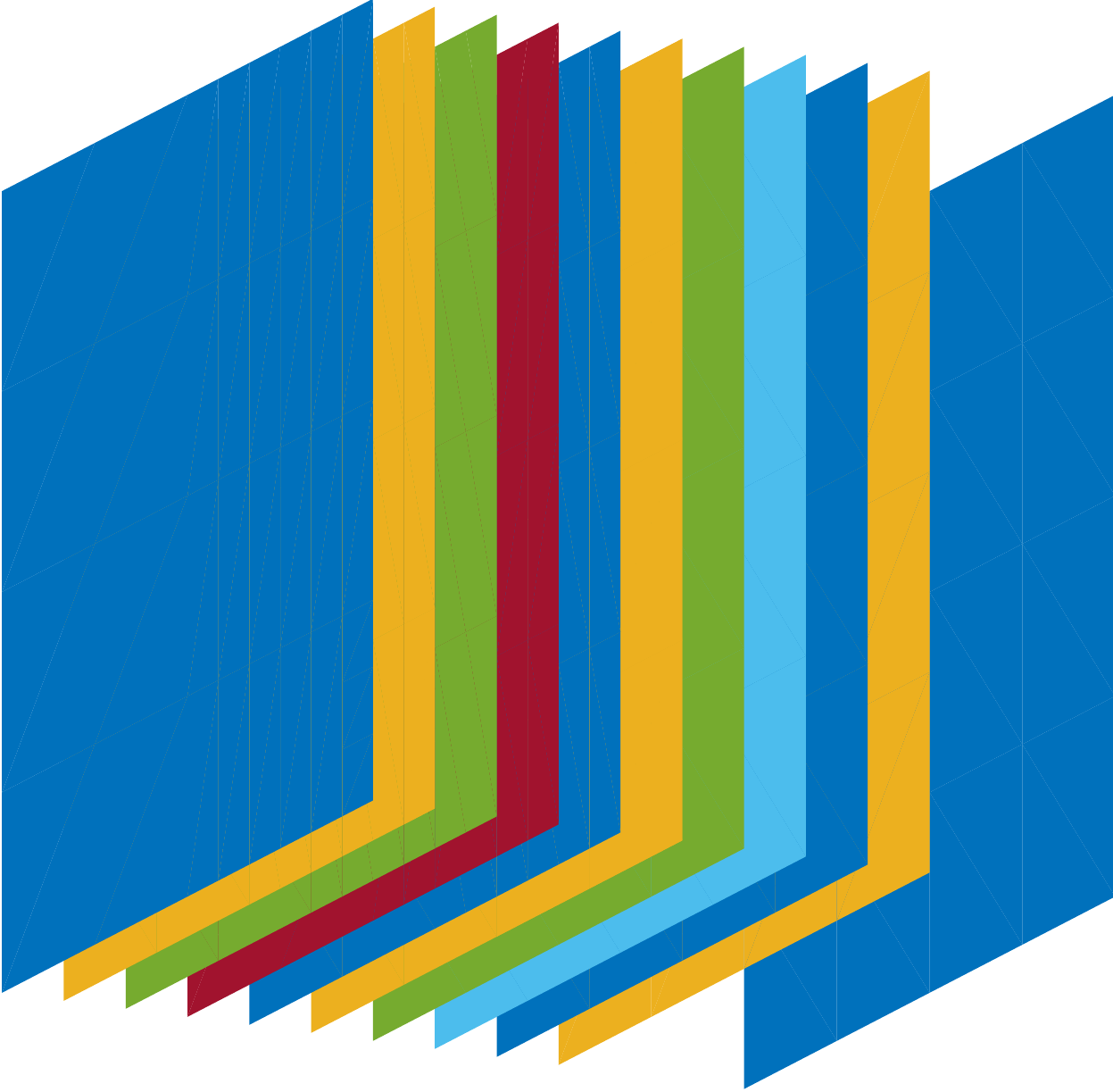}\hspace{.08\textwidth}
\includegraphics[width=.275\textwidth]{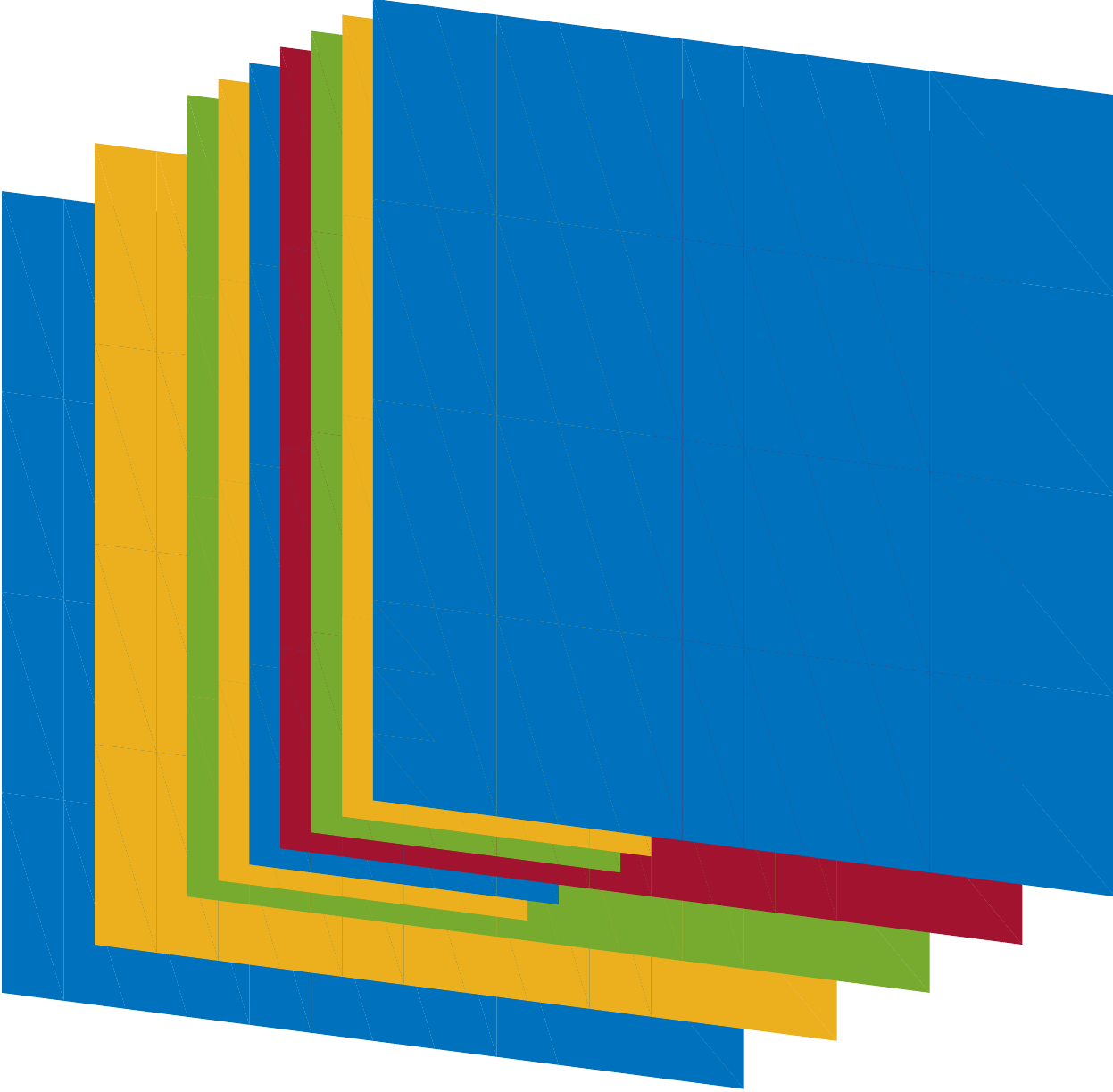}\hspace{.08\textwidth}
\includegraphics[width=.275\textwidth]{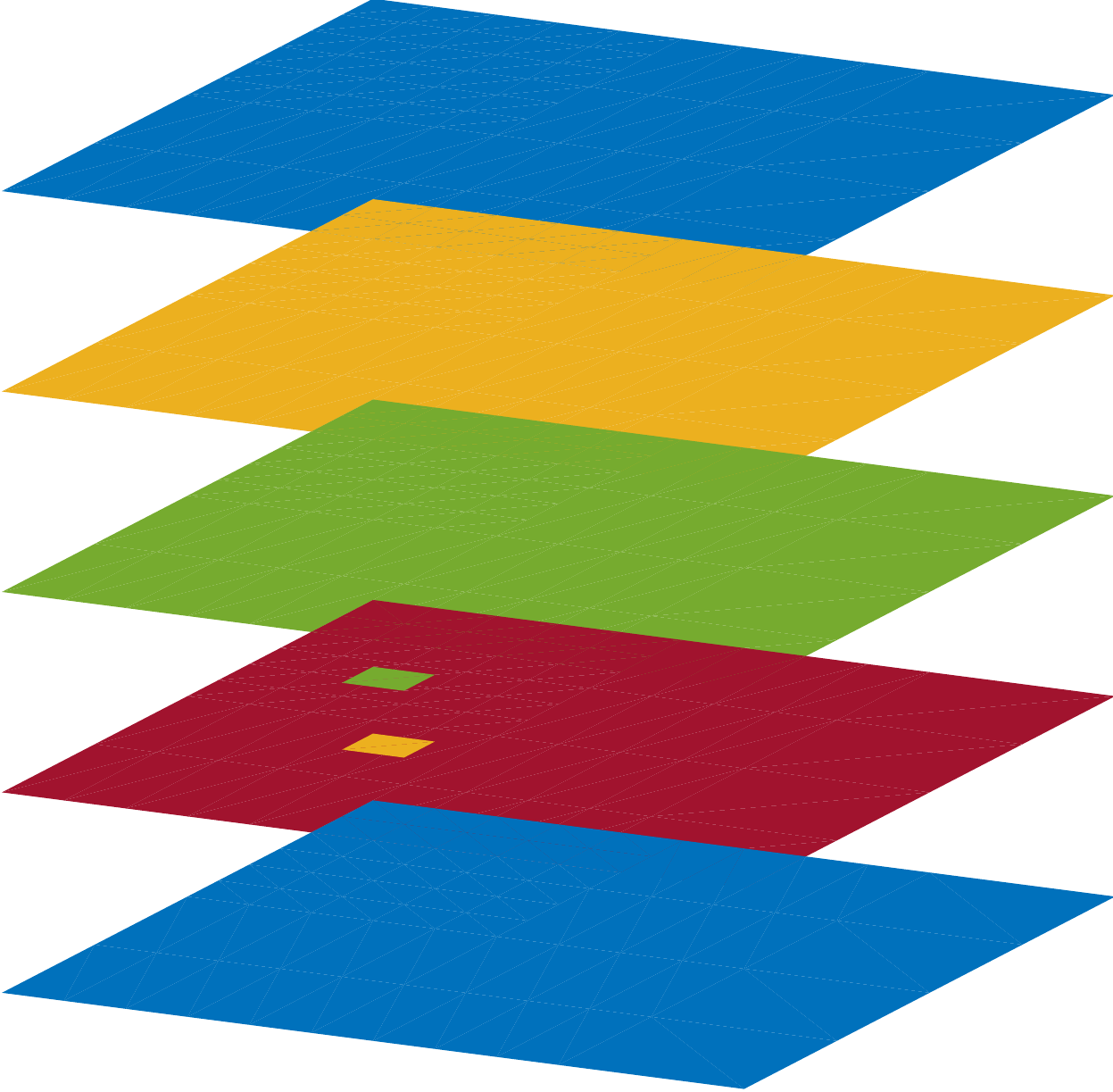}
\caption{$x$-orthogonal, $y$-orthogonal and $z$-orthogonal skeleton of the final mesh from Figure~\ref{fig: refinement algorithm}.}
\label{fig: Xi_xyz}
\end{figure}

\begin{df}[Global index sets]
For any $x,y,z\in\reell$, we define
\begin{alignat*}{2}
\XX(y,z) &\sei \bigl\{\tilde x\in[0,&\tilde X]&\mid(\tilde x,y,z)\in\xix\bigr\},\\
\YY(x,z) &\sei \bigl\{\tilde y\in[0,&\tilde Y]&\mid(x,\tilde y,z)\in\xiy\bigr\},\\
\ZZ(x,y) &\sei \bigl\{\tilde z\in[0,&\tilde Z]&\mid(x,y,\tilde z)\in\xiz\bigr\}.
\end{alignat*}
Note that in an admissible mesh, the entries $\bigl\{ 0,\dots,\ceilfrac {p_1}2-1,\enspace \tilde X-\ceilfrac {p_1}2+1,\dots,\tilde X\bigr\}$ are always included in $\XX(y,z)$ (and analogously for $\YY(x,z)$ and $\ZZ(x,y)$).
\end{df}

\begin{df}[Local index vectors]
To each active node $v=(v_1,v_2,v_3)\in\N_A$, we associate  a local index vector $\xx(v)\in\reell^{p_1+2}$, which is obtained by taking the unique $p_1+2$ consecutive elements in $\XX(v_2,v_3)$ having $v_1$ as their $\tfrac{p_1+3}2$-th (this is, the middle) entry. We analogously define $\yy(v)\in\reell^{p_2+2}$ and $\zz(v)\in\reell^{p_3+2}$.
\end{df}
\begin{df}[T-spline blending function]
 We associate to each active node $v\in\N_A$ a trivariate B-spline function, referred as \emph{T-spline blending function}, defined as the product of the B-spline functions on the corresponding local index vectors, \[B_v(x,y,z)\coloneqq N_{\xx(v)}(x) \cdot N_{\yy(v)}(y)\cdot N_{\zz(v)}(z).\]
\end{df}

\section{Analysis-Suitability}\label{sec: AS}
In this section, we give an abstract definition of Analysis-Suitability. Instead of using T-junction extensions as in the 2D case, we define \emph{perturbed regions} through the intersection of particular T-spline supports. Analysis-Suitability is then defined as the absence of intersections between these perturbed regions. This idea is comparable to the 2D case, where Analysis-Suitability is defined as the absence of intersections between T-junction extensions.
Subsequent to these definitions, we prove that all previously defined admissible meshes are analysis-suitable.

\begin{df}[Perturbed regions]\label{df: perturbed regions}%
For $q,r,s\in\reell$ define the slices
\begin{align*}
\Sx(q) &\sei\left\{(\tilde x,\tilde y,\tilde z)\in\AR\mid \tilde x=q\right\},\\
\Sy(r) &\sei\left\{(\tilde x,\tilde y,\tilde z)\in\AR\mid \tilde y=r\right\},\\
\Sz(s) &\sei\left\{(\tilde x,\tilde y,\tilde z)\in\AR\mid \tilde z=s\right\}.
\intertext{Moreover, we denote by}
\Nx(q) &\sei\left\{(v_1,v_2,v_3)\in\N_A\mid(q,v_2,v_3)\in\xix\right\}
\intertext{the set of all nodes of which the projection on the slice $\Sx(q)$ lies in some element's face.
Define analogously}
\Ny(r) &\sei\left\{(v_1,v_2,v_3)\in\N_A\mid(v_1,r,v_3)\in\xiy\right\},\\
\Nz(s) &\sei\left\{(v_1,v_2,v_3)\in\N_A\mid(v_1,v_2,s)\in\xiz\right\}.\\
\intertext{For any $q,r,s\in\reell$ we define \emph{slice perturbations}}
\Rx(q) &\sei \Sx(q)\cap\mbigcup{v\in\Nx(q)}\supp B_v\ \cap \mbigcup{v\in\N_A\setminus\Nx(q)}\supp B_v,\\
\Ry(r) &\sei \Sy(r)\cap\mbigcup{v\in\Ny(r)}\supp B_v\ \cap \mbigcup{v\in\N_A\setminus\Ny(r)}\supp B_v,\\
\Rz(s) &\sei \Sz(s)\cap\mbigcup{v\in\Nz(s)}\supp B_v\ \cap \mbigcup{v\in\N_A\setminus\Nz(s)}\supp B_v.
\end{align*}
The \emph{perturbed regions} $\Rx$, $\Ry$, $\Rz$ are defined by 
\[\Rx\sei\bigcup_{q\in\reell}\Rx(q),\quad\Ry\sei\bigcup_{r\in\reell}\Ry(r),\quad\Rz\sei\bigcup_{s\in\reell}\Rz(s).\]
In a uniform mesh, the perturbed regions are empty. In a non-uniform mesh, the perturbed regions are a superset of all hanging nodes and edges (this is, all kinds of 3D T-junctions). See Figure~\ref{fig: AS example 1} for a 2D visualization of these definitions.
\end{df}

\begin{df}[Analysis-suitability]\label{df: AS}%
A given mesh $\G$ is \emph{analysis-suitable} if the above-defined perturbed regions do not intersect, i.e. if
\[\Rx\cap\Ry =\Ry \cap \Rz=\Rz\cap \Rx= \emptyset.\]
The set of analysis-suitable meshes is denoted by $\AS$.
\end{df}
\begin{rem}
When applied in the two-dimensional case, the above definitions may yield perturbed regions that are larger than the T-junction extensions from \cite{ZSHS:2012, BBCS:2012} (see Fig.~\ref{fig: AS example 2}). However, this occurs only in meshes that are not analysis-suitable, and 
the 2D version of Definition~\ref{df: AS} is, regarding refinements of tensor-product meshes, equivalent to the classical definition of analysis-suitability%
.
\end{rem}

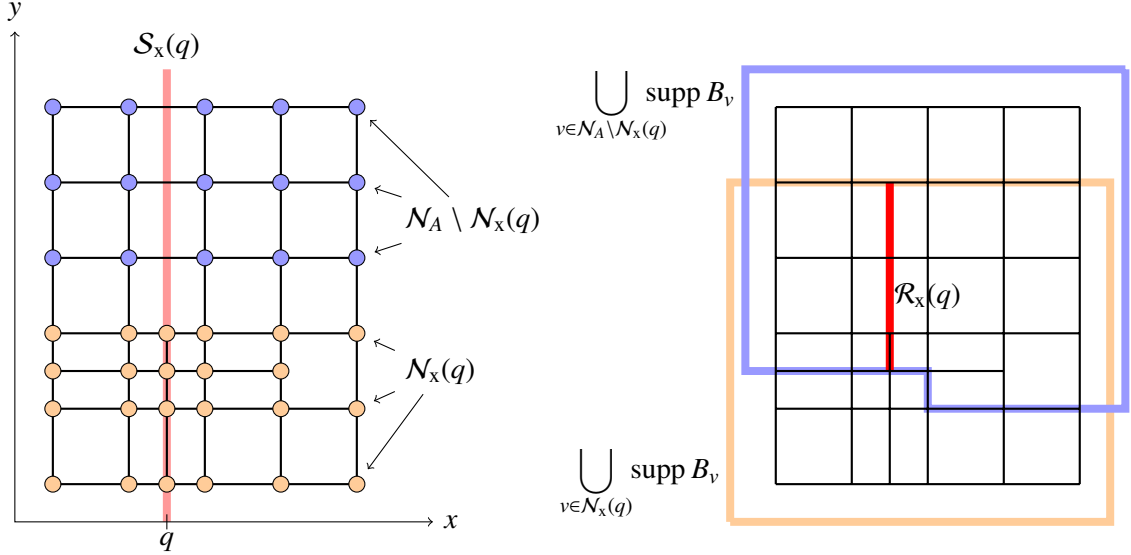
\begin{figure}[ht]
\centering
\begin{tikzpicture}[baseline=0]
\draw[line width=3pt, red!40] (1.5,-.5) coordinate (q) -- (1.5,5.5) coordinate (sx);
\node[above] at (sx) {$\Sx(q)$};
\node[below] at (q) {$q$};
\draw[<->] (-.5,6) node [above] {$y$} |- (5,-.5) node [right] {$x$};
\draw (q)++(0,-.1)--++(0,.2);
\draw[thick] (0,0) grid (4,5)
      (0,1.5)--(3,1.5) (1.5,0)--(1.5,2);
\foreach \a in {0,...,4}
{ \foreach \b in {0,1,2}
    \draw[fill=orange!40] (\a,\b) circle (3pt);
  \foreach \b in {3,4,5}
    \draw[fill=blue!40] (\a,\b) circle (3pt);
}
\foreach \b in {0,1,2}
  \draw[fill=orange!40] (1.5,\b) circle (3pt);
\foreach \b in {0,1,1.5,2,3}
  \draw[fill=orange!40] (\b,1.5) circle (3pt);
\node[right] (na) at (4.5,3.5) {$\N_A\setminus\Nx(q)$};
\node[right] (nx) at (4.5,1.5) {$\Nx(q)$};
\foreach \a in {3,4,5}
  \draw[->,shorten >= 7pt] (na) -- (4,\a);
\foreach \a in {0,1,2}
  \draw[->,shorten >= 7pt] (nx) -- (4,\a);
\end{tikzpicture}
\begin{tikzpicture}[baseline=0]
\draw[line width=3pt, orange!40] (-.6,-.5) |- (4.4,4) |- (-.6,-.5);
\draw[line width=3pt, blue!40] (4.6,5.5) |- (2,1) |- (-.4,1.5) |- (4.6,5.5);
\draw[line width=3pt,red] (1.5,1.5)--(1.5,4) ++(.5,-1.5) coordinate (rx);
\node at (rx) {$\Rx(q)$};
\draw[thick] (0,0) grid (4,5)
      (0,1.5)--(3,1.5) (1.5,0)--(1.5,2);
\node[left] at (-.6,0) {$\mbigcup{v\in\Nx(q)}\supp B_v$};
\node[left] at (-.4,5) {$\mbigcup{v\in\N_A\setminus\Nx(q)}\supp B_v$};
\end{tikzpicture}

\caption{2D example for the construction of the slice perturbation $\Rx(q)$ in an analysis-suitable mesh. The left figure illustrates the construction of $\Nx(q)$ and its complement $\N_A\setminus\Nx(q)$, and the right figure shows the resulting slice perturbation, which coincides with the corresponding classical T-junction extension. }
\label{fig: AS example 1}
\end{figure}
\begin{figure}[ht]
\centering
\begin{tikzpicture}[baseline=0]
\draw[line width=3pt, red!40] (1.5,-.5) coordinate (q) -- (1.5,5.5) coordinate (sx);
\node[above] at (sx) {$\Sx(q)$};
\node[below] at (q) {$q$};
\draw[<->] (-.5,6) node [above] {$y$} |- (5,-.5) node [right] {$x$};
\draw (q)++(0,-.1)--++(0,.2);
\draw[thick] (0,0) grid (4,5)
      (0,1.5)--(2,1.5) (1.5,0)--(1.5,2);
\foreach \a in {0,...,4}
{ \foreach \b in {0,1,2}
    \draw[fill=orange!40] (\a,\b) circle (3pt);
  \foreach \b in {3,4,5}
    \draw[fill=blue!40] (\a,\b) circle (3pt);
}
\foreach \b in {0,1,2}
  \draw[fill=orange!40] (1.5,\b) circle (3pt);
\foreach \b in {0,1,1.5,2}
  \draw[fill=orange!40] (\b,1.5) circle (3pt);
\node[right] (na) at (4.5,3.5) {$\N_A\setminus\Nx(q)$};
\node[right] (nx) at (4.5,1.5) {$\Nx(q)$};
\foreach \a in {3,4,5}
  \draw[->,shorten >= 7pt] (na) -- (4,\a);
\foreach \a in {0,1,2}
  \draw[->,shorten >= 7pt] (nx) -- (4,\a);
\end{tikzpicture}
\begin{tikzpicture}[baseline=0]
\draw[line width=3pt, orange!40] (-.6,-.5) |- (4.4,4) |- (-.6,-.5);
\draw[line width=3pt, blue!40] (4.6,5.5) |- (1,1) |- (-.4,1.5) |- (4.6,5.5);
\draw[line width=3pt,red] (1.5,1)--(1.5,4) ++(.5,-1.5) coordinate (rx);
\node at (rx) {$\Rx(q)$};
\draw[thick] (0,0) grid (4,5)
      (0,1.5)--(2,1.5) (1.5,0)--(1.5,2);
\node[left] at (-.6,0) {$\mbigcup{v\in\Nx(q)}\supp B_v$};
\node[left] at (-.4,5) {$\mbigcup{v\in\N_A\setminus\Nx(q)}\supp B_v$};
\end{tikzpicture}
\caption{2D example for the construction of the slice perturbation $\Rx(q)$ in a mesh that is not analysis-suitable. The left figure illustrates the construction of $\Nx(q)$ and its complement $\N_A\setminus\Nx(q)$, and the right figure shows the resulting slice perturbation, which is strictly larger than the corresponding classical T-junction extension. }
\label{fig: AS example 2}
\end{figure}
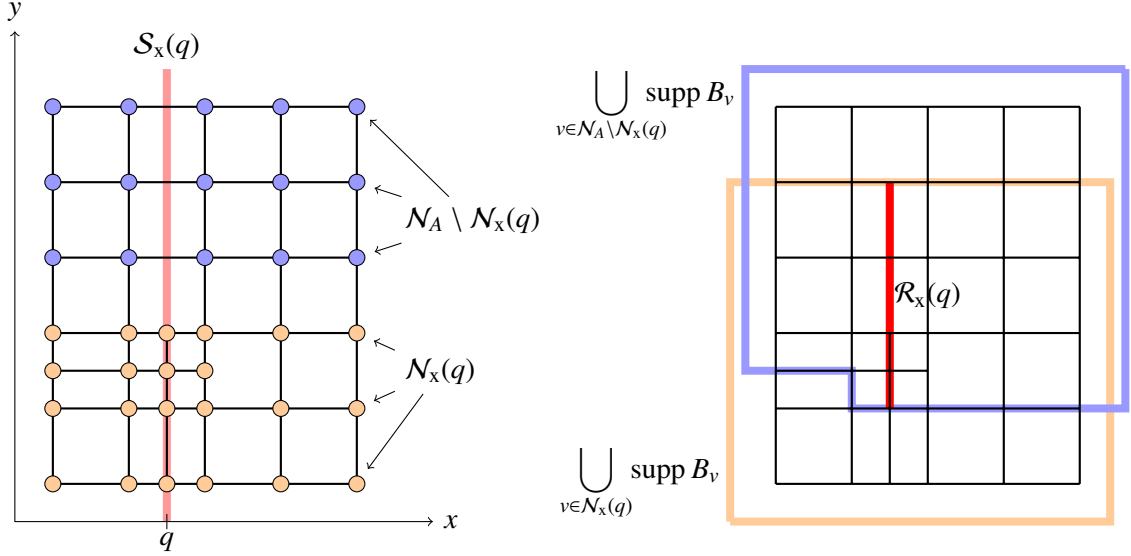

\begin{thm}\label{thm: A in AS}%
$\A\subseteq\AS$ for all $m\ge 2$.
\end{thm}

\begin{proof}
We prove the claim by induction over admissible subdivisions.
Assume $K\in\G\in\A\cap\AS$ and let $\hat\G\sei\subdiv(\G,K)\in\A$ be an admissible subdivision of $\G$. 
We have to show that $\hat\G\in\AS$. We assume without loss of generality that $\ell(K)=0\bmod 3$. Hence subdividing $K$ adds $m-1$  faces to the mesh, which are $x$-orthogonal. Set $K\eqqcolon[x,x+\tilde x]\times[y,y+\tilde y]\times[z,z+\tilde z]$ and 
$\tXi\sei \{x+\tfrac jm\tilde x\mid j\in\{1,\dots, m-1\}\}$, then the skeletons of $\hat\G$ are given by
\begin{equation*}
\hatxix=\xix\cup\tXi\times[y,y+\tilde y]\times[z,z+\tilde z] ,\quad
\hatxiy=\xiy,\quad\hatxiz=\xiz.
\end{equation*}
Let $\hat v\in\hat\N_A\setminus\N_A$ be a new active node.
Using the local quasi-uniformity from Lemma~\ref{lma: levels change slowly}, it can be verified that for all $r\in\reell$ such that $\hat v\in\hat\Ny(r)$ follows $\Ry(r)\cap\supp B_{\hat v}=\emptyset$. Consequently, $\hat\Ry=\Ry$ and analogously $\hat \Rz=\Rz$. Moreover, $\hat\Rx(q)=\Rx(q)$ for all $q\notin\tXi$. It remains to characterize 
\[\hat\Rx(\xi) = \Sx(\xi)\cap\mbigcup{v\in\hat\Nx(\xi)}\supp B_v\cap\mbigcup{v\in\hat\N_A\setminus\hat\Nx(\xi)}\supp B_v\]
for any $\xi\in\tXi$.
With 
\begin{equation}\label{eq: A in AS: active nodes}
\begin{aligned}
\hat\Nx(\xi)&=\Nx(\xi)\cup\hat\N_A\setminus\N_A\\\text{and}\enspace\hat\N_A\setminus\hat\Nx(\xi)&=\N_A\setminus\hat\Nx(\xi)=\N_A\setminus\Nx(\xi),
\end{aligned}
\end{equation}
it follows
\begin{align*}
\hat\Rx(\xi) &=
\Sx(\xi)\cap\mbigcup{v\in\hat\Nx(\xi)}\supp B_v\ \cap \mbigcup{v\in\hat\N_A\setminus\hat\Nx(\xi)}\supp B_v,\\
&\Stackrel{\eqref{eq: A in AS: active nodes}}= \Sx(\xi)\cap\Bigl(\mbigcup{v\in\Nx(\xi)}\supp B_v\cup\mbigcup{v\in\hat\N_A\setminus\N_A}\supp B_v\Bigr)\cap\mbigcup{v\in\N_A\setminus\Nx(\xi)}\supp B_v\\
&= \Rx (\xi) \cup \Bigl(\underbrace{\Sx(\xi)\ \cap\mbigcup{\hat v\in\hat\N_A\setminus\N_A}\supp B_{\hat v}}_\Sigma\ \cap\mbigcup{v\in\N_A\setminus\Nx(\xi)}\supp B_v\Bigr).
\end{align*}
We will prove below that $\Sigma\cap\hat\Rz=\Sigma\cap\hat\Ry=\emptyset$. See Figures \ref{fig: Sx-levels} and \ref{fig: Sx-perturbations} for an example with $\ell(K)=3$ and $m=2$.
Assume for contradiction that there is $s\in\reell$ with $\hat\Rz(s)\cap\Sigma\neq\emptyset$. Then there exist $v\in\hat\Nz(s)$ and $w\in\hat\N_A\setminus\hat\Nz(s)$ such that 
\begin{equation}\label{eq: A in AS: eqn to contradict}
\Sz(s)\cap \supp B_v\cap\supp B_w\cap\Sigma\neq\emptyset.
\end{equation}
Since the subdivision of $K$ is admissible, we know that all elements in $\patch\G K$ are at least of level $\ell(K)$. This implies that all those elements are of equal or smaller size than $K$. 
Denote $\midp(K)\eqqcolon(\sigma,\nu,\tau)$ and $\varepsilon\sei\tfrac{m^{-\ell(K)/3}}2$.
It follows
\begin{equation}\label{eq: A in AS: Sigma in patch}
\Sigma\subseteq\tcup\patch\G K,
\end{equation}
and with \[\hat\N_A\setminus\N_A\ \subset\ [\sigma-\varepsilon,\sigma+\varepsilon]\times[\nu-\varepsilon,\nu+\varepsilon]\times[\tau-\varepsilon,\tau+\varepsilon],\]
we get more precisely
\begin{equation}\label{eq: A in AS: def Sigma}
\Sigma\subseteq\left\{\xi\right\}\times\bigl[\nu-\varepsilon(p_2+2),\nu+\varepsilon(p_2+2)\bigr]\times\bigl[\tau-\varepsilon(p_3+2),\tau+\varepsilon(p_3+2)\bigr].
\end{equation}
The second-order patch $\patch\G{\patch\G K}\sei\bigcup_{K'\in\patch\G K}\patch\G{K'}$ consists of elements that may be larger in $z$-direction, but are of same or smaller size than $K$ in $x$- and $y$-direction. For $w=(w_1,w_2,w_3)$, Equation \eqref{eq: A in AS: eqn to contradict} implies $\supp B_w\cap\Sigma\neq\emptyset$,  and  we conclude from \eqref{eq: A in AS: def Sigma} that 
\begin{equation}\label{eq: A in AS: w is near Sigma}
(w_1,w_2)\ \in\ \bigl[\xi-\varepsilon(p_1+1),\xi+\varepsilon(p_1+1)\bigr]\times\bigl[\nu-\varepsilon(2p_2+3),\nu+\varepsilon(2p_2+3)\bigr]
\end{equation}
We assume that there is no element in $\G$ with level higher than $\ell(K)+1$. This is an eligible assumption, since every admissible mesh can be reproduced by a sequence of level-increasing admissible subdivisions; see \cite[Proposition 4.3]{Morgenstern:Peterseim:2015} for a detailed construction. This assumption implies that the $z$-orthogonal skeleton $\xiz$ is a subset of the $z$-orthogonal skeleton of a uniform $(\ell(K)+1)$-leveled mesh, 
\begin{equation}
\xiz(\G)\subseteq\xiz(\Guni{\ell(K)+1}), \label{eq: A in AS: xiz in xiz-uni}  
\end{equation}
and with $\min\ell(\patch\G K)=\ell(K)$, we have even equality on the patch $\patch\G K$,
\begin{equation}\label{eq: A in AS: xiz-patch is xiz-uni-patch}
\xiz\bigl(\patch\G K\bigr)=\xiz\bigl(\patch{\Guni{\ell(K)}}K\bigr)=\xiz\bigl(\patch{\Guni{\ell(K)+1}}K\bigr), 
\end{equation}
using the notation $\xiz\bigl(\patch\G K\bigr)\sei\xiz(\G)\cap\tcup\patch\G K$. 
Since $v\in\hat\Nz(s)$, we know that \mbox{$\hat\Nz(s)\neq\emptyset$}, 
which means that there are elements in $\G$ that have $z$-orthogonal faces at the $z$-coordinate $s$, i.e., 
\mbox{$\Sz(s)\cap \xiz(\G) \neq\emptyset$}. With \eqref{eq: A in AS: xiz in xiz-uni} we get $\Sz(s)\cap \xiz(\Guni{\ell(K)+1}) \neq\emptyset$.
Since $\Guni{\ell(K)+1}$ is a tensor-product mesh, its $z$-orthogonal skeleton consists of global domain slices, which yields
$\Sz(s)\subseteq \xiz(\Guni{\ell(K)+1}).$
The restriction to the patch $\patch\G K$ yields
\begin{equation}\label{eq: A in AS: Szs-GpK in xiz-GpK}
\Sz(s)\cap\tcup\patch\G K\subseteq \xiz\bigl(\patch{\Guni{\ell(K)+1}}K\bigr) \stackrel{\eqref{eq: A in AS: xiz-patch is xiz-uni-patch}}=\xiz\bigl(\patch\G K\bigr)\subseteq\xiz(\G).
\end{equation}
Equation \eqref{eq: A in AS: eqn to contradict} implies that $\Sz(s)\cap\Sigma\neq\emptyset$, and
with \eqref{eq: A in AS: Sigma in patch} we get that $\Sz(s)\cap\tcup\patch\G K\neq\emptyset$. 
Hence
\begin{align}
\Sz(s)\cap\tcup\patch\G K\ &\supseteq\ \Sz(s)\cap\U(K) \notag \\
&=\ \bigl[\xi-\varepsilon(2p_1+3),\xi+\varepsilon(2p_1+3)\bigr]\times\bigl[\nu-\varepsilon(2p_2+3),\nu+\varepsilon(2p_2+3)\bigr]\times\{s\}.
\end{align}
Since $w\notin\hat\Nz(s)$, we know by definition that $(w_1,w_2,s)\notin\xiz$. Then it follows from \eqref{eq: A in AS: Szs-GpK in xiz-GpK} that $(w_1,w_2,s)\notin\Sz(s)\cap\tcup\patch\G K$,
and hence 
\begin{equation}
(w_1,w_2)\ \notin\ \bigl[\xi-\varepsilon(2p_1+3),\xi+\varepsilon(2p_1+3)\bigr]\times\bigl[\nu-\varepsilon(2p_2+3),\nu+\varepsilon(2p_2+3)\bigr] 
\end{equation}
in contradiction to \eqref{eq: A in AS: w is near Sigma}. This proves that $\hat\Rz\cap\Sigma=\emptyset$. Similar arguments prove that $\Sigma\cap\Ry=\emptyset$, which concludes the proof.
\end{proof}

\begin{figure}[!ht]
\centering
\begin{tikzpicture}[scale=.65]
\sf\scriptsize
\fill[blue!30] (6,9.5) rectangle (10.5,12);
\draw[ultra thick, blue!70] (5,7) rectangle (11.5,14);
\draw (0,0) grid (17,21);
\foreach \a in {5,...,11}
  \draw (\a+.5,7)--(\a+.5,14);
\foreach \a in {9.5,10.5,11.5}
  \draw (6,\a)--(10.5,\a);
\foreach \a in {0,...,20}
{ \foreach \b in {0,1,2,14,15,16}
    \node at (\b+.5,\a+.5) {0};
}
\foreach \a in {3,...,13}
{ \foreach \b in {0,1,2,18,19,20}
    \node at (\a+.5,\b+.5) {0};
}
\foreach \a in {3,...,17}
{ \foreach \b in {3,4,12,13}
    \node at (\b+.5,\a+.5) {1};
}
\foreach \a in {5,...,11}
{ \foreach \b in {3,4,5,6,14,15,16,17}
    \node at (\a+.5,\b+.5) {1};
}
\foreach \a in {7,...,13}
{ \foreach \b in {5,5.5,10.5,11,11.5}
    \node at (\b+.25,\a+.5) {2};
}
\foreach \a in {12,...,20}
{ \foreach \b in {7,8,12,13}
    \node at (\a/2+.25,\b+.5) {2};
}
\foreach \a in {12,...,20}
{ \foreach \b in {9,9.5,10,11,11.5}
    \node at (\a/2+.25,\b+.25) {3};
}
\foreach \a in {12,13,14,15,17,18,19,20}
  \node at (\a/2+.25,10.75) {3};
\node at (8.25,10.75) {4};
\end{tikzpicture}
\caption{$yz$-view on the slice $\Sx(\xi)$. The numbers denote element levels, and the element in the center with level 4 is a child of $K$. The patch $\patch{\hat\G}K$ is highlighted in blue, and the second-order patch $\patch{\hat\G}{\patch{\hat \G}K}$ is indicated by a thick blue line.}
\label{fig: Sx-levels}
\end{figure}

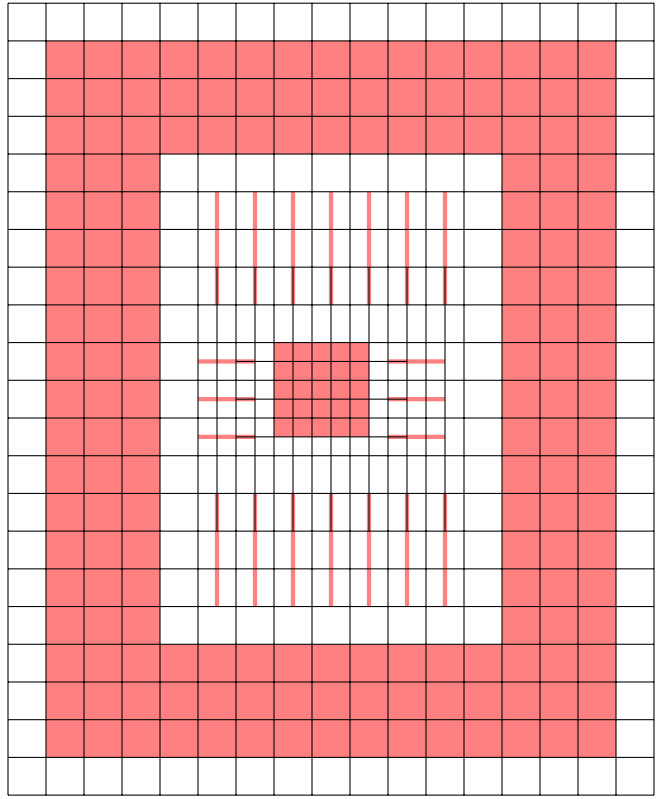
\begin{figure}[!ht]
\centering
\begin{tikzpicture}[scale=.5]
\fill[red!50] (7,9.5) rectangle (9.5,12)
(1,1) -- (16,1) -- (16,20) -- (1,20) -- (1,1) -- (4,4) -- (4,17) -- (13,17) -- (13,4) -- (4,4) -- (1,1);
\foreach \a in {5,...,11}
  \draw[ultra thick, red!50] (\a+.5,5)--(\a+.5,8)  (\a+.5,13)--(\a+.5,16);
\foreach \a in {9.5,10.5,11.5}
  \draw[ultra thick, red!50]  (5,\a)--(6.5,\a)  (10,\a)--(11.5,\a);
\draw (0,0) grid (17,21);
\foreach \a in {5,...,11}
  \draw (\a+.5,7)--(\a+.5,14);
\foreach \a in {9.5,10.5,11.5}
  \draw (6,\a)--(10.5,\a);
\end{tikzpicture}
\caption{$yz$-view on the slice $\Sx(\xi)$. $\Rx$ is indicated by red areas. $\Ry$ is depicted by horizontal red lines, $\Rz$ are vertical red lines. At the same time, the squared red area in the center coincides with $\Sigma$.}
\label{fig: Sx-perturbations}
\end{figure}
\section{Dual-Compatibility}\label{sec: DC}%
This section recalls the concept of Dual-Compatibility, which is a sufficient criterion for linear independence of the T-spline functions, based on dual functionals. We follow the ideas of \cite{BBSV:2014} for the definitions and for the proof of linear independence. In addition, we prove that all analysis-suitable (and hence all admissible) meshes are dual-compatible and thereby generalize a 2D result from \cite{BBCS:2012}.

\begin{prp}[Dual functional, {\cite[Theorem 4.41]{Schumaker:2007}}]\label{prp:overlap}%
Given the local index vector $X=(x_1,\dots,x_{p+2})$, there exists an $L^2$-functional $\lambda_{X}$ with $\supp\lambda_X=\supp N_X$ such that for any $\tilde X=(\tilde x_1,\dots,\tilde x_{p+2})$ satisfying
\begin{equation}\label{eq: overlap}
 \begin{alignedat}{4}
  \forall\ x&\in\{x_1,\dots,x_{p+2}\}&&:&\quad \tilde x_1&\leq x\leq \tilde x_{p+2}&&\Rightarrow x\in \{\tilde x_1,\dots,\tilde x_{p+2}\}\\
  \text{and}\quad\forall\ \tilde x&\in\{\tilde x_1,\dots,\tilde x_{p+2}\}&&:&\quad x_1&\leq \tilde x\leq x_{p+2}&&\Rightarrow \tilde x\in \{x_1,\dots,x_{p+2}\},
 \end{alignedat}
\end{equation}
follows $\lambda_{X}(N_{\tilde X}) = \delta_{X\tilde X}$.
\end{prp}
\begin{proof}
Following \cite{Schumaker:2007}, we construct a dual functional
on the same local knot vector $X$ which we denote by $\lambda_{X}:L^2\bigl([0,1]\bigr)\to\reell$. For details, see \cite[Theorem 4.34, 4.37, and 4.41]{Schumaker:2007}.
Let $y_j=\cos\bigl(\tfrac{p-j+1}{p+1}\pi\bigr)$ for $j=0,\dots,p+1$.
Using divided differences, the perfect B-spline of order $p+1$ is defined by
\[B^*_{p+1}(x)\sei (p+1)\,(-1)^{p+1}\bigl[y_0,\dots,y_{p+1}\bigr]\left((x-\bullet)_+\right)^p\]
and satisfies (amongst other things) $\int_{-1}^1B^*_{p+1}(x)\de x=1$ as depicted in Figure~\ref{pic: perfect B-spline}. 
\begin{figure}[tb]
 \centering
 \includegraphics[width=.5\textwidth]{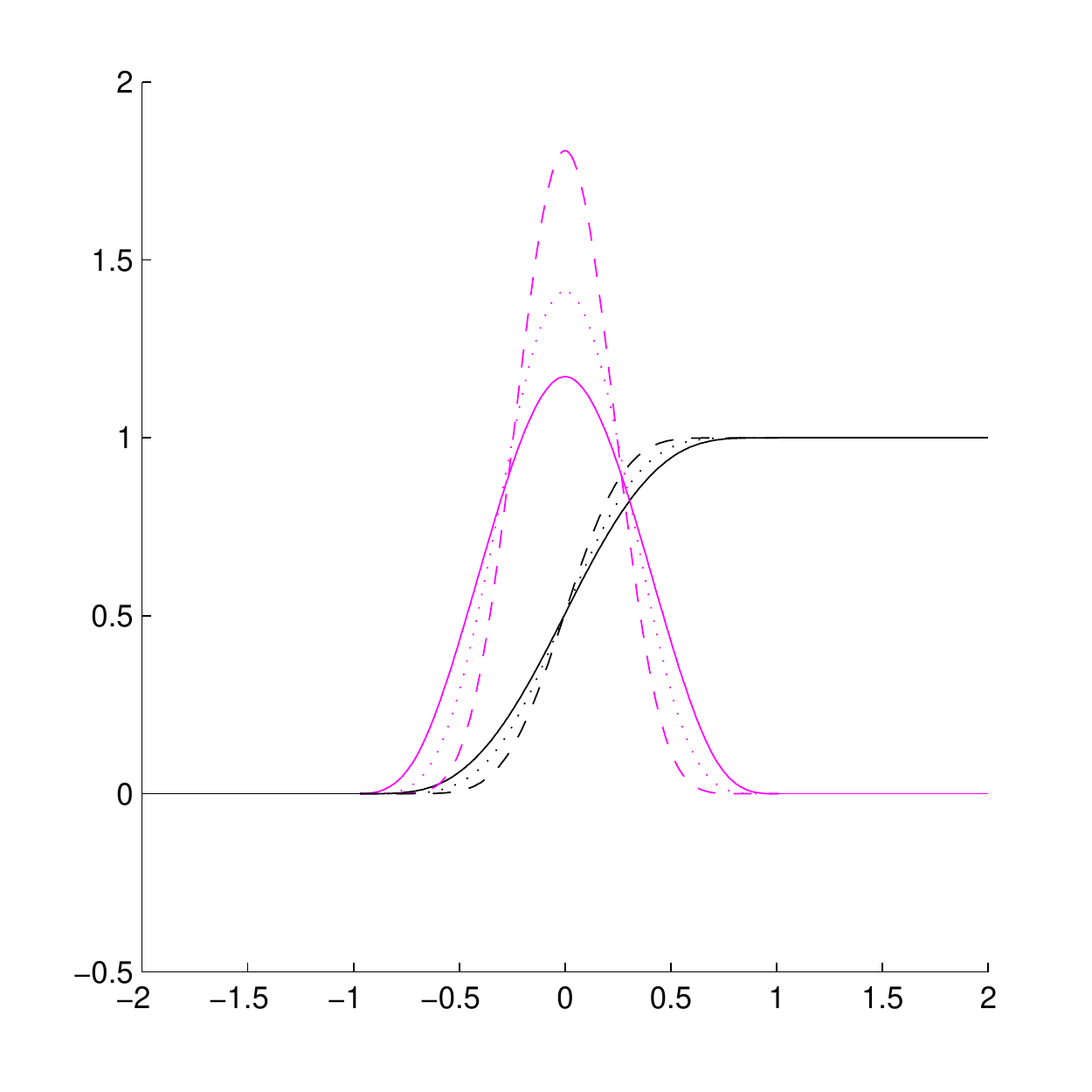}
 \caption{Plot of the perfect B-splines $B^*_4$ (solid), $B^*_6$ (dotted), $B^*_{10}$ (dashed) and the corresponding antiderivatives.}
 \label{pic: perfect B-spline}
\end{figure}
Set \[G_{X}(x) \sei \int_{-1}^{\tfrac{2x-{x_1}-{x_{p+2}}}{{x_{p+2}}-{x_1}}}B^*_{p+1}(t)\de t \quad\text{for }{x_1}\le x\le {x_{p+2}}\] and
\[\phi_{X}(x)=\tfrac1{p!}\left(x-{x_2}\right)\cdots\bigl(x-{x_{p+1}}\bigr).\] We define the dual functional by 
\begin{equation}\label{eq:def_lambda}
\lambda_{X}(f) = \int_{{x_1}}^{{x_{p+2}}}\mspace{-9mu}f\, D^{p+1}(G_{X}\,\phi_{X})\de x\quad\text{for all }f\in L^2\bigl([0,1]\bigr).
\end{equation}
Note in particular that for all $f\in L^2(\reell)$ with $f|_{[x_1,x_{p+2}]}=0$ follows $\lambda_{X}(f)=0$.
If \eqref{eq: overlap} holds then the claim follows by construction, see \cite[Theorem 4.41]{Schumaker:2007}.
\end{proof}
We say that two index vectors verifying \eqref{eq: overlap} \emph{overlap}. 
In order to define the set of T-spline blending functions of which we desire linear independence, we construct local index vectors for each active node.

\begin{df}
We define the functional $\lambda_v$ by
 \[\lambda_v(B_{w})\sei \lambda_{\xx(v)}(N_{\xx(w)})\cdot\lambda_{\yy(v)}(N_{\yy(w)})\cdot\lambda_{\zz(v)}(N_{\zz(w)})\]
 using the one-dimensional functional $\lambda_X$ defined in \eqref{eq:def_lambda}.
\end{df}

\begin{df}\label{df: partial overlap}%
We say that a couple of nodes $v,w\in\N$ \emph{partially overlap} if their index vectors overlap in at least two out of three dimensions; this is, if (at least) two of the pairs \[\bigl(\xx(v),\xx(w)\bigr),\ \bigl(\yy(v),\yy(w)\bigr),\ \bigl(\zz(v),\zz(w)\bigr)\] overlap in the sense of Proposition~\ref{prp:overlap}.\bigskip
\end{df}

\begin{df}
A mesh $\G$ is \emph{dual-compatible (DC)} if any two active nodes $v,w\in\N_A$ with $\bigl|\supp B_v \cap\supp B_w\bigr|>0$ partially overlap.
The set of dual-compatible meshes is denoted by $\DC$.
\end{df}
\begin{rem}
The above Definition~\ref{df: partial overlap} fulfills the definition of \emph{partial overlap} given in \cite[Def.\ 7.1]{BBSV:2014}, which is not equivalent. The definition given in \cite{BBSV:2014} is more general, and the corresponding mesh classes are nested in the sense 
$\DC\subseteq\DC_{\text{\cite{BBSV:2014}}}$. However, we do have equivalence of these definitions in the two-dimensional setting.
\end{rem}

The following lemma states that the perturbed regions from Definition~\ref{df: perturbed regions} indicate non-overlapping knot vectors, and it is applied in the proof of Theroem~\ref{thm: AS in DC} below.
\begin{lma}\label{lma: perturbed regions indicate non-overlapping}%
Let  $q\in[0,\tilde X]$ and $v_1,v_2\in\N_A$.
If $v_1\in\Nx(q)\notni v_2$  and $\Sx(q)\cap\supp B_{v_1}\cap\supp B_{v_2}\neq\emptyset$, then $\xx(v_1)$ and $\xx(v_2)$ do not overlap in the sense of \eqref{eq: overlap}.
\end{lma}
This holds analogously for $\Ny(r),r\in[0,\tilde Y]$ and $\Nz(s),s\in[0,\tilde Z]$.
\begin{proof}
Let $v_1=(x_1,y_1,z_1)$. From $v_1\in\Nx(q)$ and Definition~\ref{df: perturbed regions}, we conclude that $(q,y_1,z_1)\in\xix$, and hence $q\in\XX(y_1,z_1)$. Let $\xx(v_1)=(x_1^1,\dots,x_1^{p_1+2})$ be the local $x$-direction knot vector associated to $v_1$, then \mbox{$\supp B_{v_1}\cap \Sx(q)\neq\emptyset$} implies that $x_1^1\le q\le x_1^{p_1+2}$. This and  $q\in\XX(y_1,z_1)$ yield $q\in\xx(v_1)$.
Let \mbox{$v_2=(x_2,y_2,z_2)$}. From $v_2\notin\Nx(q)$, we get $(q,y_2,z_2)\notin\xix$, hence $q\notin\XX(y_2,z_2)$, and in particular \mbox{$q\notin\xx(v_2)$}. Let \mbox{$\xx(v_2)=(x_2^1,\dots,x_2^{p_1+2})$} be the local knot vector associated to $v_2$, then \mbox{$\supp B_{v_2}\cap \Sx(q)\neq\emptyset$} implies that $x_2^1\le q\le x_2^{p_1+2}$. Together with $\xx(v_1)\ni q\notin\xx(v_2)$, we see that $v_1$ and $v_2$ do not overlap.
\end{proof}

\begin{thm}\label{thm: AS in DC}%
$\AS=\DC$.
\end{thm}

\begin{proof}
``$\subseteq$''.
Assume for contradiction a mesh $\G$ which is not DC, hence there exist active nodes \mbox{$v,w\in\N_A$} with \mbox{$\bigl|\supp B_v\cap\supp B_w\bigr|>0$} that do not overlap in two dimensions, without loss of generality  $x$ and $y$.
We will show that there exist two slice perturbations $\Rx(q)$ and $\Ry(r)$ with nonempty intersection.
We denote $v=(v_1,v_2,v_3)$, $w=(w_1,w_2,w_3)$ and $\xx(v)=(x^v_1,\dots,x^v_{p_1+2})$. The elements of $\yy(v),\xx(w),\yy(w)$ are denoted analogously. Moreover we define
\begin{alignat*}{2}
x\mn&\sei\max(x^v_1,x^w_1),&\enspace x\mx&\sei\min(x^v_{p_1+2},x^w_{p_1+2})\\
y\mn&\sei\max(y^v_1,y^w_1),& y\mx&\sei\min(y^v_{p_2+2},y^w_{p_2+2})\\
z\mn&\sei\max(z^v_1,z^w_1),& z\mx&\sei\min(z^v_{p_3+2},z^w_{p_3+2})
\end{alignat*}
and note that
\begin{align*}
\supp B_v \cap  \supp B_w = [x\mn,x\mx]\times[y\mn,y\mx]\times[z\mn,z\mx].
\end{align*}
Since $\xx(v)$ and $\xx(w)$ do not overlap, there exists $q\in[x\mn,x\mx]$ with either  \mbox{$\xx(v)\ni q\notin\xx(w)$}\enspace or \mbox{$\xx(v)\notni q\in\xx(w)$}. Without loss of generality we assume $\xx(v)\ni q\notin\xx(w)$. Since $\xx(v)\subseteq \XX(v_2,v_3)$, it follows by definition that $(q,v_2,v_3)\in\xix$ and hence $v\in\Nx(q)$.
Since $q\notin\xx(w)$ and hence  $(q,w_2,w_3)\notin\xix$, it follows that $w\notin\Nx(q)$.
Then
\begin{align*}
\Rx(q) &= \Sx(q)\ \cap\  \mathbox{ \bigcup_{v'\in\Nx(q)} } \supp B_{v'}\ \cap\  \mathbox{ \bigcup_{v'\in\N_A\smallsetminus\Nx(q)} } \supp B_{v'}\\
 &\supseteq \Sx(q)\ \cap \ \supp B_v\ \cap \ \supp B_w\\
 &= \{q\}\times[y\mn,y\mx]\times[z\mn,z\mx].
 \intertext{Analogously, we have}
 \Ry(r) &\supseteq [x\mn,x\mx]\times\{r\}\times[z\mn,z\mx]
 \intertext{and hence}
 \Rx(q) \cap \Ry(r) &\supseteq \{q\}\times\{r\}\times[z\mn,z\mx]\neq\emptyset,
 \end{align*}
which means that the mesh $\G$ is not analysis-suitable.\bigskip

``$\supseteq$''.
Assume for contradiction that the mesh is not analysis-suitable, and w.l.o.g.\ that there is \mbox{$v=(q,r,s)\in\reell^3$} such that $\Rx\cap\Ry\supseteq\{v\}\neq\emptyset$. Definition~\ref{df: perturbed regions} implies that there exist
$v_1,v_2,v_3,v_4\in\N_A$ with
\mbox{$v_1\in\Nx(q)\notni v_2$} and \mbox{$v_3\in\Ny(r)\notni v_4$}\enspace such that
\[v\ \in\ \Sx(q)\cap\Sy(r)\cap\supp B_{v_1}\cap\supp B_{v_2}\cap\supp B_{v_3}\cap\supp B_{v_4}.\]
Lemma~\ref{lma: perturbed regions indicate non-overlapping} yields that $\xx(v_1)$ and $\xx(v_2)$ do not overlap, and that $\yy(v_3)$ and $\yy(v_4)$ do not overlap.

\textit{Case 1.} If $v_1\in\Ny(r)\notni v_2$, or $v_1\notin\Ny(r)\ni v_2$, then $v_1$ and $v_2$ do not partially overlap.

\textit{Case 2.} If $v_1\in\Ny(r)$ and $v_4\notin\Nx(q)$, then $v_1$ and $v_4$ do not partially overlap.

\textit{Case 3.} If $v_1\notin\Ny(r)$ and $v_3\notin\Nx(q)$, then $v_1$ and $v_3$ do not partially overlap.

\textit{Case 4.} If $v_2\in\Ny(r)$ and $v_4\in\Nx(q)$, then $v_2$ and $v_4$ do not partially overlap.

\textit{Case 5.} If $v_2\notin\Ny(r)$ and $v_3\in\Nx(q)$, then $v_2$ and $v_3$ do not partially overlap.

In all cases (see the table in Fig.~\ref{tb: cases}), the mesh is not dual-compatible. This concludes the proof.
\end{proof}
\begin{figure}[ht]
\renewcommand \arraystretch 1
\newcommand \Hline {\hhline{*{3}{-|}-||-}}
\centering
\begin{tabular}{c|c|c|c||l}
$v_1\in\Ny(r)$ & $v_2\in\Ny(r)$ & $v_3\in\Ny(r)$ & $v_4\in\Ny(r)$ & case(s) \\\hhline{*{3}{=:}=::=}
\true & \true & \true & \true & 4 \\\Hline
\true & \true & \true & \false & 2 \\\Hline
\true & \true & \false & \true & 4 \\\Hline 
\true & \true & \false & \false & 2 \\\Hline 
\true & \false & \true & \true & 1, 5 \\\Hline 
\true & \false & \true & \false & 1, 2, 5 \\\Hline 
\true & \false & \false & \true & 1 \\\Hline 
\true & \false & \false & \false & 1, 2 \\\Hline 
\false & \true & \true & \true & 1, 4 \\\Hline 
\false & \true & \true & \false & 1 \\\Hline 
\false & \true & \false & \true & 1, 3, 4 \\\Hline 
\false & \true & \false & \false & 1, 3 \\\Hline 
\false & \false & \true & \true & 5 \\\Hline 
\false & \false & \true & \false & 5 \\\Hline 
\false & \false & \false & \true & 3 \\\Hline 
\false & \false & \false & \false & 3
\end{tabular}
\caption{The five cases considered in the proof of Theorem~\ref{thm: AS in DC} cover all possible configurations.}
\label{tb: cases}
\end{figure}

\begin{thm}\label{thm: DC has dual basis}
 Let $\G$ be a DC T-mesh. Then the set of functionals $\{\lambda_v\mid v\in\N_A\}$ is a set of dual functionals for the set $\{B_v\mid v\in\N_A\}$.
\end{thm}
The proof below follows the ideas of \cite[Proposition~5.1]{BBCS:2012} and \cite[Proposition~7.3]{BBSV:2014}.
\begin{proof}
Let $v,w\in\N_A$. We need to show that 
\begin{equation}\label{eq:DC_claim}
\lambda_v(B_w) = \delta_{vw},
\end{equation}
with $\delta$ representing the Kronecker symbol.

If $\supp B_v$ and $\supp B_w$ are disjoint (or have an intersection of empty interior), then at least one of the pairs 
\[\bigl(\supp(N_{\xx(v)}),\supp(N_{\xx(w)})\bigr),\ \bigl(\supp(N_{\yy(v)}),\supp(N_{\yy(w)})\bigr),\ \bigl(\supp(N_{\zz(v)}),\supp(N_{\zz(w)})\bigr)\]
has an intersection with empty interior. Assume w.l.o.g.\ that $\left|\supp(N_{\xx(v)})\cap\supp(N_{\xx(w)})\right|=0$, then 
\[\lambda_v(B_w) = \underbrace{\lambda_{\xx(v)}(N_{\xx(w)})}_0\cdot \lambda_{\yy(v)}(N_{\yy(w)})\cdot \lambda_{\zz(v)}(N_{\zz(w)})=0.\]
Assume that $\supp B_v$ and $\supp B_w$ have an intersection with nonempty interior.
Since the mesh $\G$ is DC, the two nodes overlap in at least two dimensions. Without loss of generality we may assume the index vectors $\bigl(\xx(v),\xx(w)\bigr)$ and $\bigl(\yy(v),\yy(w)\bigr)$ overlap. Proposition~\ref{prp:overlap} yields 
\[\lambda_{\xx(v)}(N_{\xx(w)}) = \delta_{v_1w_1} \enspace\text{and}\enspace \lambda_{\yy(v)}(N_{\yy(w)}) = \delta_{v_2w_2} .\]
The above identities immediately prove \eqref{eq:DC_claim} if $v_1\neq w_1$ or $v_2\neq w_2$. If on the contrary, $v_1=w_1$ and $v_2\neq w_2$, then $v$ and $w$ are aligned in $z$-direction, this is, $\zz(v)$ and $\zz(w)$ are both vectors of $p+2$ consecutive indices from the same index set 
$\ZZ(v_1,v_2)=\ZZ(w_1,w_2)$. 
Hence $v$ and $w$ must overlap also in $z$-direction. Again, Proposition~\ref{prp:overlap} yields
\[\lambda_{\zz(v)}(N_{\zz(w)}) = \delta_{v_3w_3},\]
which concludes the proof.
\end{proof}

\begin{crl}[{\cite[Proposition~7.4]{BBSV:2014}}]
Let $\G$ be a DC T-mesh. Then the set $\{B_v\mid v\in\N_A\}$ is linear independent.
\end{crl}
\begin{proof}
Assume \[\sum_{v\in\N_A}c_vB_v=0\] for some coefficients $\{c_v\}_{v\in\N_A}\subseteq\reell$. Then, for any $w\in\N_A$, applying $\lambda_w$ to the sum, using linearity and Theorem~\ref{thm: DC has dual basis}, we get
\[c_w = \lambda_w\,\Bigl(\sum_{v\in\N_A}c_vB_v\Bigr)=0.\]
\raiseqed
\end{proof}

\section{Linear Complexity}\label{sec: complexity}
This section is devoted to a complexity estimate in the style of a famous estimate for the Newest Vertex Bisection on triangular meshes given by Binev, Dahmen and DeVore \cite{BDV:2004} and, in an alternative version, by Stevenson~\cite{Stevenson:2007}.
Linear Complexity of the refinement procedure is an inevitable criterion for optimal convergence rates in the Adaptive Finite Element Method (see e.g.\ \cite{BDV:2004,Stevenson:2007,CFPP:2014} and \cite[Conclusions]{Buffa:Giannelli:2015}). The estimate and its proof follow our own work \cite{Morgenstern:Peterseim:2015,BGMP:2016}, which we generalize now to three dimensions and $m$-graded refinement.
The estimate reads as follows.

\begin{thm}\label{thm: complexity}
Any sequence of admissible meshes $\G_0,\G_1,\dots,\G_J$ with \[\G_j=\refine(\G_{j-1},\M_{j-1}),\quad\M_{j-1}\subseteq\G_{j-1}\quad\text{for}\enspace j\in\{1,\dots,J\}\] satisfies
\[\left|\G_J\setminus\G_0\right|\ \le\ C_{\mathbf p,m}\sum_{j=0}^{J-1}|\M_j|\ ,\]
with $C_{\mathbf p,m}=\tfrac{m^{1/3}}{1-m^{-1/3}}\,\bigl(4d_1+1\bigr)\,\bigl(4d_2+m^{1/3}\bigr)\,\bigl(4d_3+m^{2/3}\bigr)$ and $d_1,d_2,d_3$ from Lemma~\ref{lma: K1 in refMS => K2 in S} below.
\end{thm}

\begin{lma}\label{lma: K1 in refMS => K2 in S}
Given $\M\subseteq\G\in\A$ and $K\in\refine(\G,\M)\setminus\G$, there exists $K'\in \M$ such that $\ell(K)\le\ell(K')+1$ and \[\Dist(K,K')\le m^{-\ell(K)/3}(d_1,d_2,d_3),\]
with ``$\le$'' understood componentwise and constants 
\begin{align*}
d_1&\sei \tfrac1{1-m^{-1/3}}        \,\bigl(p_1+\tfrac{3+m^{1/3}}2+\tfrac{m^{1/3}-1}{m^2}\bigr),\\
d_2&\sei \tfrac{m^{1/3}}{1-m^{-1/3}}\,\bigl(p_2+\tfrac{3+m^{1/3}}2+\tfrac{m^{1/3}-1}{m^2}\bigr),\\
d_3&\sei \tfrac{m^{2/3}}{1-m^{-1/3}}\,\bigl(p_3+\tfrac{3+m^{1/3}}2+\tfrac{m^{1/3}-1}{m^2}\bigr).
\end{align*}
\end{lma}
The proof is given in Appendix~\ref{apx: K1 in refMS => K2 in S}.

\begin{proof}[Proof of Theorem~\ref{thm: complexity}]\ 

\pnumpx For $K\in\tcup\A$ and $\KM\in\M\sei\M_0\cup\dots\cup\M_{J-1}$, define $\lambda(K,\KM)$ by \[\lambda(K,\KM)\sei\begin{cases}m^{(\ell(K)-\ell(\KM))/3}&\text{if }\ell(K)\le\ell(\KM)+1\text{ and }\Dist(K,\KM)\le 2m^{-\ell(K)/3}(d_1,d_2,d_3),\\[.3em]0&\text{otherwise.}\end{cases}\]

\pnumpx[Main idea of the proof.]
\begin{alignat*}{2}
\left|\G_J\setminus\G_0\right| &= \mathbox[2.5em]{\sum_{K\in\G_J\setminus\G_0}}1 &&\Stackrel{\numref{sum_lambda > 1}}\le\sum_{K\in\G_J\setminus\G_0}\sum_{\KM\in\M}\lambda(K,\KM) \\
&\Stackrel{\numref{summe aller lambdas beschraenkt}}\le\sum_{\KM\in\M} C_{\mathbf p,m} &&=C_{\mathbf p,m}\,\sum_{j=0}^{J-1} |\M_j|.
\end{alignat*}

\pnumpx[Each $K\in\G_J\setminus\G_0$ satisfies \[\sum_{\KM\in\M}\lambda(K,\KM)\ \ge\ 1.\]]%
\label{sum_lambda > 1}%
Consider $K\in\G_J\setminus\G_0$. Set $j_1<J$ such that $K\in\G_{j_1+1}\setminus\G_{j_1}$. Lemma~\ref{lma: K1 in refMS => K2 in S} states the existence of $K_1\in\M_{j_1}$ with $\Dist(K,K_1)\le m^{-\ell(K)/3}(d_1,d_2,d_3)$ and $\ell(K)\le\ell(K_1)+1$. Hence $\lambda(K,K_1)=m^{\ell(K)-\ell(K_1)}>0$.
The repeated use of Lemma~\ref{lma: K1 in refMS => K2 in S} yields $j_1>j_2>j_3>\dots$ and $K_2,K_3,\dots$ with $K_{i-1}\in\G_{j_i+1}\setminus\G_{j_i}$ and $K_i\in\M_{j_i}$ such that 
\begin{equation}\label{eq: complexity -last}
\Dist(K_{i-1},K_i)\le m^{-\ell(K_{i-1})/3}(d_1,d_2,d_3)\enspace\text{and}\enspace\ell(K_{i-1})\le\ell(K_i)+1.
\end{equation}
We repeat applying Lemma~\ref{lma: K1 in refMS => K2 in S} as $\lambda(K,K_i)>0$ and $\ell(K_i)>0$, and we stop at the first index $L$ with $\lambda(K,K_L)=0$ or $\ell(K_L)=0$. 
If $\ell(K_L)=0$ and $\lambda(K,K_L)>0$, then
\[\sum_{\KM\in\M}\lambda(K,\KM)\ge\lambda(K,K_L)=m^{(\ell(K)-\ell(K_L))/3}\ge m^{1/3}.\]
If $\lambda(K,K_L)=0$ because $\ell(K)>\ell(K_L)+1$, then \eqref{eq: complexity -last} yields $\ell(K_{L-1})\le\ell(K_L)+1<\ell(K)$ and hence
\[\sum_{\KM\in\M}\lambda(K,\KM)\ge\lambda(K,K_{L-1})=m^{(\ell(K)-\ell(K_{L-1}))/3}>m^{1/3}.\]
If $\lambda(K,K_L)=0$ because $\Dist(K,K_L)>2m^{-\ell(K)/3}(d_1,d_2,d_3)$, then a triangle inequality shows
\begin{align*}
2m^{-\ell(K)/3}(d_1,d_2,d_3) &< \Dist(K,K_1)+\sum_{i=1}^{L-1}\Dist(K_i,K_{i+1}) 
\\&\le\ m^{-\ell(K)/3}(d_1,d_2,d_3)+\sum_{i=1}^{L-1} m^{-\ell(K_i)/3}(d_1,d_2,d_3),
\end{align*}
and hence $\smash{\displaystyle m^{-\ell(K)/3} \le\sum_{i=1}^{L-1} m^{-\ell(K_i)/3}}$. The proof is concluded with 
\[ 1\ \le\ \sum_{i=1}^{L-1} m^{(\ell(K)-\ell(K_i))/3}\ =\ \sum_{i=1}^{L-1} \lambda(K,K_i)\ \le\ \sum_{\KM\in\M}\lambda(K,\KM).\]

\pnumpx[For all $j\in\{0,\dots,J-1\}$ and $\KM\in\M_j$ holds \[\sum_{K\in\G_J\setminus\G_0}\lambda(K,\KM)\ \le\ \tfrac{m^{1/3}}{1-m^{-1/3}}\,\bigl(4d_1+1\bigr)\,\bigl(4d_2+m^{1/3}\bigr)\,\bigl(4d_3+m^{2/3}\bigr)\ =\ C_{\mathbf p,m}\ .\]]%
\label{summe aller lambdas beschraenkt}%
This is shown as follows. By definition of $\lambda$, we have
\begin{align*}
\mathbox[1cm]{\sum_{K\in\G_J\setminus\G_0}}\lambda(K,\KM)
&\le \mathbox[1cm]{\sum_{K\in\bigcup\A\setminus\G_0}}\lambda(K,\KM)\\
&= \sum_{j=1}^{\ell(\KM)+1}m^{(j-\ell(\KM))/3}\,\#\underbrace{\bigl\{K\in\tcup\A\mid\ell(K)=j\text{ and }\Dist(K,\KM)\le 2m^{-j/3}(d_1,d_2,d_3)\bigl\}}_B.
\end{align*}
Since we know by definition of the level that $\ell(K)=j$ implies $|K|=m^{-j}$, we know that $m^j\left|\tcup B\right|$ is an upper bound of $\#B$. The cuboidal set $\tcup B$ is the union of all admissible elements of level $j$ having their midpoints inside a cuboid of size 
\[4m^{-j/3}d_1\,\times\,4m^{-j/3}d_2\,\times\,4m^{-j/3}d_3.\]
An admissible element of level $j$ is not bigger than $m^{-j/3}\,\times\,m^{(1-j)/3}\,\times\,m^{(2-j)/3}$. 
Together, we have \[\bigl|\tcup B\bigr|\le m^{-j}\,\bigl(4d_1+1\bigr)\,\bigl(4d_2+m^{1/3}\bigr)\,\bigl(4d_3+m^{2/3}\bigr),\] and hence $\#B\le \bigl(4d_1+1\bigr)\,\bigl(4d_2+m^{1/3}\bigr)\,\bigl(4d_3+m^{2/3}\bigr)$. An index substitution $k\sei1-j+\ell(\KM)$ proves the claim with
\[\sum_{j=1}^{\ell(\KM)+1}m^{(j-\ell(\KM))/3}=\sum_{k=0}^{\ell(\KM)}m^{(1-k)/3}<m^{1/3}\sum_{k=0}^\infty m^{-k/3}=\tfrac{m^{1/3}}{1-m^{-1/3}}.\]\raiseqed
\end{proof}
\vspace{1em}
\subsection*{An experiment on $C_{\mathbf p,m}$}

The constant $C_{\mathbf p,m}$ arising from this theory is very large, however we observed much smaller ratios of refined and marked elements in the experiment (in all cases less than $\tfrac{C_{\mathbf p,m}}{3000}$, see Figure~\ref{fig: complexity constants}). Starting from a $5\times5\times5$ mesh, we applied the refinement algorithm with only one corner element marked, always sticking to the same corner. This is realistic when resolving a singularity of the solution of a discretized PDE. The advantage of greater grading parameters could not be seen in random refinement all over the domain.

\begin{figure}[ht]
\centering
\includegraphics[width=.45\textwidth]{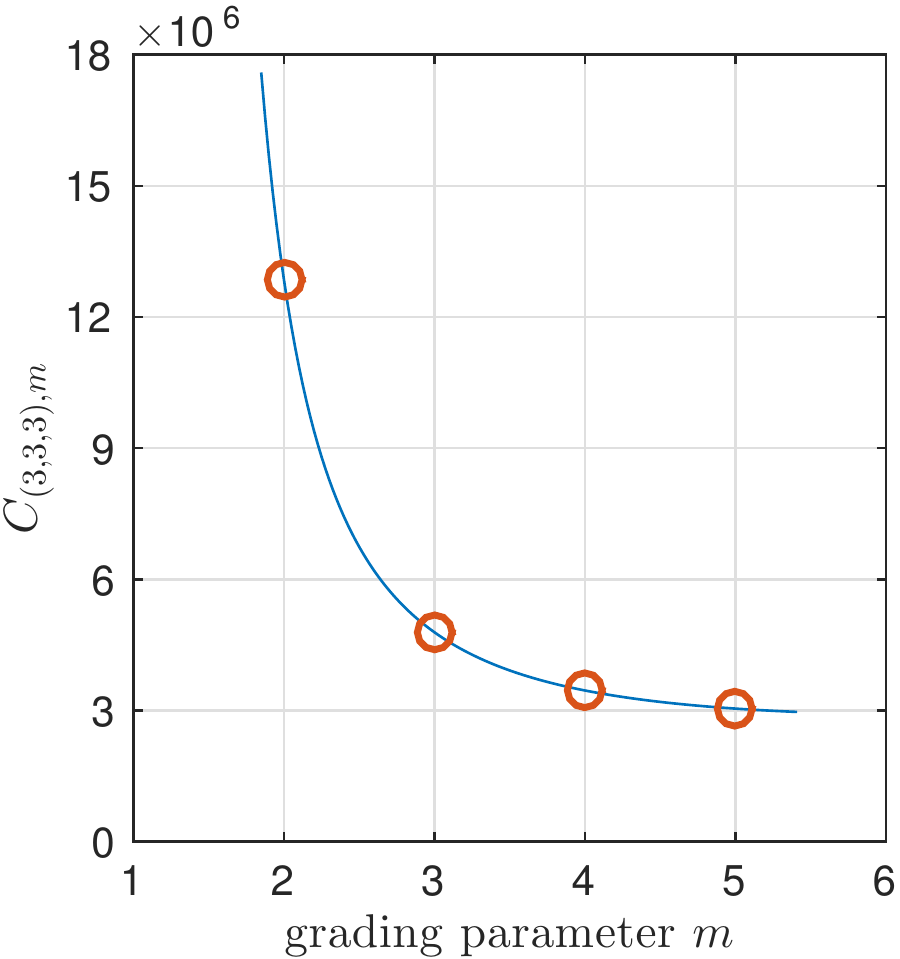}%
\hspace{.1\textwidth}%
\includegraphics[width=.45\textwidth]{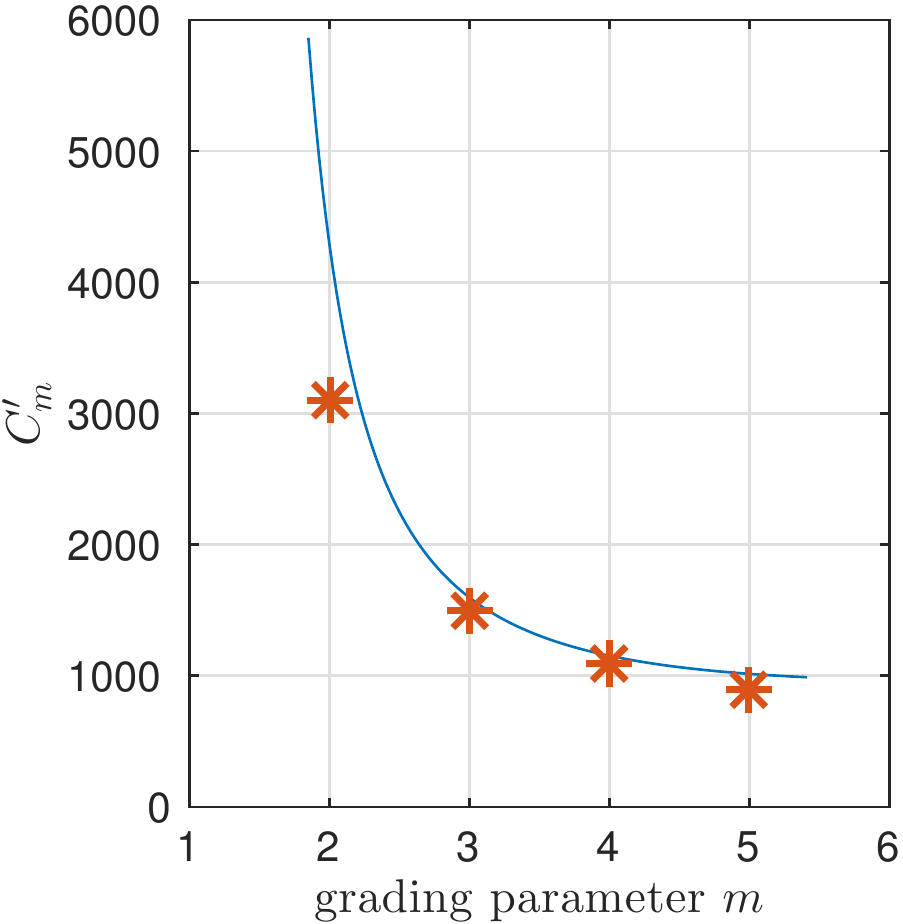}
\caption{The complexity constant $C_{\mathbf p,m}$ in theory (left) and experiment (right). The values of $C'_m$ were taken from an experiment illustrated in Figure~\ref{fig: testing complexity}.}
\label{fig: complexity constants}
\end{figure}

\begin{figure}[ht]
\centering
\includegraphics[width=.70\textwidth]{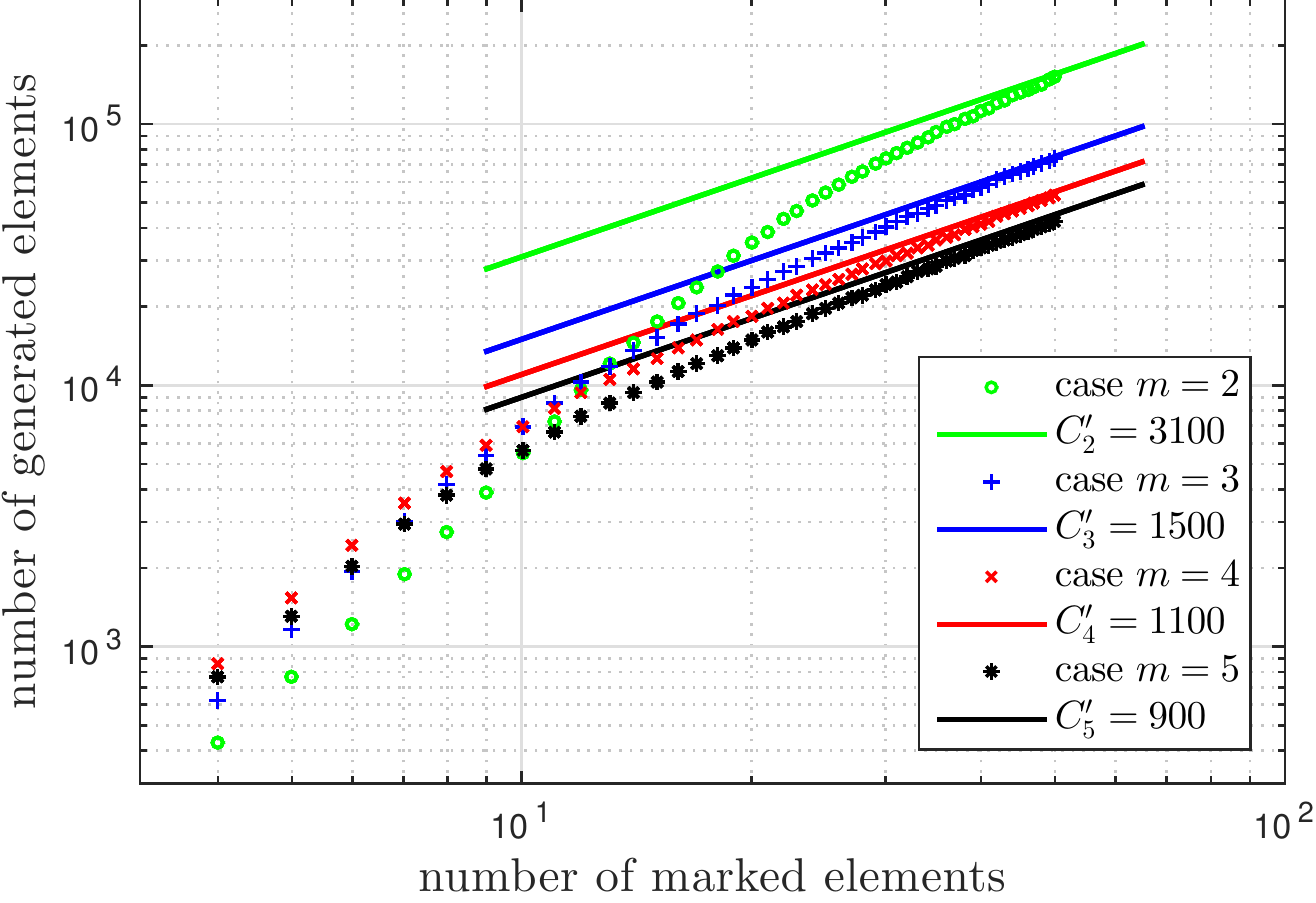}
\caption{Estimation of the experimental constants $C'_m$ for $m=2,\dots,5$.}
\label{fig: testing complexity}
\end{figure}

\section{Conclusions \& Outlook}\label{sec: conclusions}

We have generalized the concept of Analysis-Suitability to three-dimensional meshes that originate from tensor-product initial meshes, and proved that it guarantees linear independence of the T-spline blending functions.
We introduced a local refinement algorithm with adjustable mesh grading, and proved that it has linear complexity in the sense that the overhead for preserving Analysis-Suitability is essentially bounded  by the number of marked elements.
We expect that these results also generalize to even-degree and mixed-degree T-splines. In order to achieve this, a universal definition of anchor elements is needed, based on the techniques from \cite{BBSV:2013}.

Open questions that have not been investigated in this paper address the overlay (this is, the coarsest common refinement of two meshes), the nesting behavior of the T-spline spaces, and more general meshes.
As in our preliminary work \cite{Morgenstern:Peterseim:2015}, we expect that the overlay has a bounded cardinality in terms of the two overlaid meshes, and that it is also an admissible mesh.
Nestedness of T-spline spaces is not evident in general \cite{Li:Scott:2014}, but we expect nestedness for the meshes generated by the proposed refinement algorithm. A first step in this issue will be a characterization of three-dimensional meshes that induce nested T-spline spaces.
A generalization of this paper to a more general class of meshes will most likely require a manifold representation of the mesh, and use recent results on Dual-Compatibility in spline manifolds \cite{STV:2015}.
%

\providecommand{\bysame}{\leavevmode\hbox to3em{\hrulefill}\thinspace}
\providecommand{\MR}{\relax\ifhmode\unskip\space\fi MR }
\providecommand{\MRhref}[2]{%
  \href{http://www.ams.org/mathscinet-getitem?mr=#1}{#2}
}
\providecommand{\href}[2]{#2}

\appendix
\section{Minor proofs}
\subsection{Proof of Lemma~\ref{lma: magic patches are nested}}\label{apx: magic patches are nested}%
\begin{silentproof}
If $K=\hat K$, the claim is trivially fulfilled. If otherwise $K\subsetneqq\hat K$, we consider the following two cases.

\emph{Case 1.} Assume that $\ell(K)=\ell(\hat K)+1$. 
Since $K=[x,x+\tilde x]\times[y,y+\tilde y]\times[z,z+\tilde z]$ is the result of successive subdivisions of a unit cube, it holds that
\begin{equation}\label{eq: df size}
\operatorname{size}(\ell(K))\sei(\tilde x,\tilde y, \tilde z) =\begin{cases}m^{-\ell(K)/3}\left(1,1,1\right)&\text{if }\ell(K)=0\bmod 3,\\m^{-(\ell(K)-1)/3}\bigl(\tfrac1m,1,1\bigr)&\text{if }\ell(K)=1\bmod 3,\\m^{-(\ell(K)-2)/3}\bigl(\tfrac1m,\tfrac1m,1\bigr)&\text{if }\ell(K)=2\bmod 3.\end{cases}
\end{equation}
Since $K$ results from the subdivision of $\hat K$, we also have that
\begin{equation}\label{eq: parent dist}
\Dist(K,\hat K)=\begin{cases}\bigl(m^{-(\ell(\hat K)+6)/3},0,0\bigr)&\text{if }\ell(\hat K)=0\bmod 3,\\\bigl(0,m^{-(\ell(\hat K)+5)/3},0\bigr)&\text{if }\ell(\hat K)=1\bmod 3,\\\bigl(0,0,m^{-(\ell(\hat K)+4)/3}\bigr)&\text{if }\ell(\hat K)=2\bmod 3.\end{cases}
\end{equation}
Recall that
\[\D(k)\sei\begin{cases}m^{-k/3}\,\bigl(p_1+\tfrac32,p_2+\tfrac32,p_3+\tfrac32\bigr)&\text{if }k=0\bmod3,
\\[.7ex]
m^{-(k-1)/3}\,\bigl(\tfrac{p_1+3/2}m,p_2+\tfrac32,p_3+\tfrac32\bigr)&\text{if }k=1\bmod3,
\\[.7ex]
m^{-(k-2)/3}\,\bigl(\tfrac{p_1+3/2}m,\tfrac{p_2+3/2}m,p_3+\tfrac32\bigr)&\text{if }k=2\bmod3.\end{cases}
\]
We rewrite \eqref{eq: parent dist} in the form 
\begin{equation}\label{eq: child dist}
\Dist(K,\hat K)=\begin{cases}\bigl(0,0,m^{-(\ell(K)+3)/3}\bigr)&\text{if }\ell(K)=0\bmod 3,\\\bigl(m^{-(\ell(K)+5)/3},0,0\bigr)&\text{if }\ell(K)=1\bmod 3,\\\bigl(0,m^{-(\ell(K)+4)/3},0\bigr)&\text{if }\ell(K)=2\bmod 3\end{cases}
\end{equation}
and observe that $\D(\ell(K))+\Dist(K,\hat K)\le\D(\ell(K)-1)=\D(\ell(\hat K))$. The case~1 is concluded with
\begin{align*}
\U(K) &= \{x\in\Omega\mid\Dist(K,x)\le\D(\ell(K))\}\\
&\subseteq \{x\in\Omega\mid\Dist(\hat K,x)\le\D(\ell(K)) + \Dist(K,\hat K)\}\\
&\subseteq\U(\hat K),
\end{align*}
and consequently $\patch\G K\subseteq\patch\G{\hat K}$.

\emph{Case 2.} Consider $K\subset\hat K$ with $\ell(K)>\ell(\hat K)+1$, then there is a sequence 
\[K=K_0\subset K_1\subset\dots\subset K_J=\hat K\] such that $K_{j-1}\in\child(K_j)$  for $j=1,\dots,L$. Case~1 yields \[\patch\G K\subseteq\patch\G{K_1}\subseteq\dots\subseteq\patch\G{\hat K}.\]\raiseqed
\end{silentproof}

\subsection{Proof of Lemma~\ref{lma: levels change slowly}}\label{apx: levels change slowly}%
\begin{silentproof}
For $\ell(K)=0$, the assertion is always true. For $\ell(K)>0$, consider the parent $\hat K$ of $K$ (i.e., the unique element $\hat K\in\tcup\A$ with $K\in\child(\hat K)$). 
Since $\G$ is admissible, there are admissible meshes $\G_0,\dots,\G_J=\G$ and some $j\in\{0,\dots,J-1\}$ such that $K\in\G_{j+1}=\subdiv(\G_j,\{\hat K\})$.
The admissibility $\G_{j+1}\in\A$ implies that any $K'\in\patch{\G_j}{\hat K}$ satisfies $\ell(K')\ge\ell(\hat K)=\ell(K)-1$.
Since levels do not decrease during refinement, we get
\begin{align}
\ell(K)-1\le\min\ell(\patch{\G_j}{\hat K})&\le\min\ell(\patch\G{\hat K})\label{eq: levels change slooowly}\\
\notag&\Stackrel{\text{Lemma~\ref{lma: magic patches are nested}}}\le\enspace\min\ell(\patch\G K).
\end{align}
\end{silentproof}
\subsection{Proof of Lemma~\ref{lma: K1 in refMS => K2 in S}}\label{apx: K1 in refMS => K2 in S}%
\begin{silentproof}
The coefficient $\D(k)$ from Definition~\ref{df: magic patch} is bounded by 
\begin{equation}\label{eq: D bounded}
\D(k)\le m^{-k/3}\,\underbrace{\Bigl(p_1+\tfrac32,\ m^{1/3}\bigl(p_2+\tfrac32\bigr),\ m^{2/3}\bigl(p_3+\tfrac32\bigr)\Bigr)}_{\mathbf{\tilde p}}\quad\text{for all }k\in\mathbb N.
\end{equation}
Recall $\operatorname{size}(k)$ from \eqref{eq: df size} and note that it is decreasing and bounded by
\begin{equation}\label{eq: size bounded}
\operatorname{size}(k)\le m^{-k/3}\,\bigl(1,m^{1/3},m^{2/3}\bigr).
\end{equation}
Hence for $\tilde K\in\G\in\A$ and $\tilde K'\in\patch\G{\tilde K}$, there is $x\in \tilde K'\cap\U(\tilde K)$ and
hence
\begin{align}
\Dist(\tilde K,\tilde K') &\le \Dist(\tilde K,x) + \Dist(\tilde K',x) \notag\\
&\le \Dist(\tilde K,x) + \tfrac12\operatorname{size}(\ell(\tilde K')) \notag\\
&\Stackrel{\text{Lemma~\ref{lma: levels change slowly}}}\le\quad\Dist(\tilde K,x) + \tfrac12\operatorname{size}(\ell(\tilde K)-1) \notag\\
&\Stackrel{\eqref{eq: size bounded}}\le m^{-\ell(\tilde K)/3}\,\mathbf{\tilde p} +  m^{-\ell(\tilde K)/3}\underbrace{\bigl(\tfrac{m^{1/3}}2,\tfrac{m^{2/3}}2,\tfrac m2\bigr)}_{\textstyle\mathbf s} \notag\\
\label{eq: magic radius}
&\le m^{-\ell(\tilde K)/3}\left(\mathbf{\tilde p} + \mathbf s\right).
\end{align}
The existence of $K\in\refine(\G,\M)\setminus\G$ means that Algorithm~\ref{alg: refinement} subdivides $K'=K_J,K_{J-1},\dots,K_0$ such that $K_{j-1}\in\patch\G{K_j}$ and $\ell(K_{j-1})<\ell(K_j)$ for $j=J,\dots,1$, having $K'\in \M$ and $K\in\child(K_0)$, with `$\child$' from Definition~\ref{df: subdivision}. Lemma~\ref{lma: levels change slowly} yields $\ell(K_{j-1})=\ell(K_j)-1$ for $j=J,\dots,1$, which yields the estimate
\begin{alignat*}{2}
\Dist(K',K_0)\enspace&\le\enspace \sum_{j=1}^J\Dist(K_j,K_{j-1})
&\enspace &\Stackrel{\eqref{eq: magic radius}}\le \enspace \sum_{j=1}^Jm^{-\ell(K_j)/3}\,(\mathbf{\tilde p}+\mathbf s) \\[1ex]
&= \sum_{j=1}^Jm^{-(\ell(K_0)+j)/3}\,(\mathbf{\tilde p}+\mathbf s) &
&< m^{-\ell(K_0)/3}\,(\mathbf{\tilde p}+\mathbf s)\sum_{j=1}^\infty m^{-j/3} \\[1ex]
&= \frac{m^{-1/3-\ell(K_0)/3}}{1-m^{-1/3}}\,(\mathbf{\tilde p}+\mathbf s) &
&= \frac{m^{-\ell(K)/3}}{1-m^{-1/3}}\,(\mathbf{\tilde p}+\mathbf s).
\end{alignat*}
From \eqref{eq: child dist} we get 
\[\Dist(K_0,K)\le \bigl(m^{-(\ell(K)+5)/3},m^{-(\ell(K)+4)/3},m^{-(\ell(K)+3)/3}\bigr).\] 
This and a triangle inequality conclude the proof.
\end{silentproof}

\end{document}